\documentclass[a4paper,11pt,oneside]{article}
\usepackage[utf8]{inputenc}
\usepackage[T1]{fontenc}
\usepackage[english]{babel}  
\usepackage{amsthm}
\usepackage{amssymb}
\usepackage{amsmath}
\usepackage{mathrsfs}
\usepackage{graphicx}
\usepackage{xcolor}
\usepackage{eurosym}
\usepackage{dsfont}
\usepackage{pifont}
\usepackage{fancybox}
\usepackage{multicol}
\usepackage{enumitem}
\usepackage{url}
\usepackage{bbm}
\usepackage[all]{xy}
\usepackage{authblk}
\reversemarginpar
\numberwithin{equation}{section}
\usepackage[top=3.2cm, bottom=3.3cm, left=2.5cm , right=2.5cm]{geometry}
\usepackage[colorlinks=true,linkcolor=black,citecolor=black,urlcolor=black,hyperfootnotes=false]{hyperref}

\usepackage[titletoc,toc,page,title,header]{appendix}

\theoremstyle{definition} \newtheorem{Def}{Definition}[section]
\theoremstyle{plain}\newtheorem{Thm}[Def]{Theorem}
\theoremstyle{plain}\newtheorem{Prop}[Def]{Proposition}
\theoremstyle{definition}\newtheorem{Rem}[Def]{Remark}
\theoremstyle{definition}
\theoremstyle{plain}\newtheorem{Lem}[Def]{Lemma}
\theoremstyle{plain}\newtheorem{Cor}[Def]{Corollary}
\theoremstyle{remark} \newtheorem*{Claim*}{Claim}
\theoremstyle{definition} 
\theoremstyle{plain}

\newcommand{\R}{\mathbb{R}}
\newcommand{\N}{\mathbb{N}}

\newcommand{\CC}{\mathcal{C}}
\newcommand{\A}{\mathscr{A}}
\newcommand{\D}{\mathcal{D}}

\newcommand{\n}{\mathrm{n}}
\newcommand{\bb}[1]{\mathbf{#1}}
\newcommand{\Prob}{\mathbb{P}}
\newcommand{\E}[1]{\mathbb{E}\left[#1\right]}

\newcommand{\V}[1]{\mathbf{Var}{\left[#1\right]}}
\newcommand{\Cov}[2]{\mathbf{Cov}{\left[#1,#2\right]}}

\newcommand{\eqLaw}{ \overset{d}{=} }
\newcommand{\Law}{\xrightarrow[]{d} }
\newcommand{\toP}{\xrightarrow[]{\Prob} }

\newcommand{\norm}[1]{\lVert#1\rVert}
\newcommand{\scal}[2]{\langle #1,#2\rangle}

\newcommand{\eps}{\varepsilon}
\newcommand{\ind}[1]{\mathds{1}{[#1]}}

\allowdisplaybreaks

\renewcommand{\(}{\left(}
\renewcommand{\)}{\right)}
\renewcommand{\[}{\left[}
\renewcommand{\]}{\right]}
\renewcommand{\{}{\left\lbrace}
\renewcommand{\}}{\right\rbrace}

\def\wt{\widetilde}

\begin{document}

\title{\normalsize{\textbf{\uppercase{Functional convergence of Berry's nodal lengths: approximate tightness and total disorder}}}}
\date{{\footnotesize \today}}
\author{Massimo Notarnicola$^1$, Giovanni Peccati$^1$, Anna Vidotto$^2$}
\date{\footnotesize%
    {\it $^1$Department of Mathematics, University of Luxembourg}\\
    {\it $^2$Department of Mathematics and Applications, University of Naples Federico II}\\[2ex]
    \today} 

\maketitle

\abstract{We consider Berry’s random planar wave model (1977), and prove spatial functional limit theorems -- in the high-energy limit -- for discretized and truncated versions of the random field obtained by restricting its nodal length to rectangular domains. Our analysis is crucially based on a detailed study of the projection of nodal lengths onto the so-called {\it second Wiener chaos}, whose high-energy fluctuations are given by a Gaussian {\it total disorder field} indexed by polygonal curves. Such an exact characterization is then combined with moment estimates for suprema of stationary Gaussian random fields, and with a tightness criterion by Davydov and Zikitis (2005). \\
{\bf Keywords:} Central Limit Theorems; Functional Convergence; Gaussian  Fields; Nodal Sets; Random Waves;  Total Disorder. \\
{\bf AMS Classification:} 60G60, 60F05, 34L20, 33C10}

\tableofcontents

\section{Introduction}

{\it In this work, every random object is defined on a common probability space $(\Omega, \mathscr{F}, \mathbb{P} )$, with $\mathbb{E}$ denoting expectation with respect to $\mathbb{P}$.}

\subsection{The model}\label{ss:model}

The aim of this paper is to initiate the study of the {\bf high-energy functional fluctuations} of geometric quantities attached to the zero set of {\bf Berry’s Random Plane Wave model} $B_E = \{B_E (x) : x \in\R^2\}$ with parameter $E > 0$ --- thus addressing a question left open in \cite{PV:20}. Berry's model is defined as the unique (in distribution) isotropic centered Gaussian field on the plane solving almost surely the Helmholtz equation
$$
\Delta B_E(x) + 4\pi^2E \cdot B_E(x)=0, \quad x = (x_1,x_2)\in \R^2,
$$
with $\Delta=\partial^2/\partial x_1^2+\partial^2/\partial x_2^2$ equal to the usual Laplace operator\footnote{The presence of the prefactor $4\pi^2$ is inherited from \cite{NPR:19, PV:20}, where it was introduced in order to facilitate the connections with the literature about Arithmetic Random Waves \cite{KKW:13, MPRW:16}. }. Using e.g. \cite[Theorem 5.7.2]{AT:07}, the above characterization is equivalent to requiring that the covariance function of $B_E$  is given by 
\begin{equation}\label{covE}
r^E(x, y) = r^E(x-y) := J_0(2\pi \sqrt E \norm{x - y}), \quad  x,y\in \R^2,
\end{equation}
where $J_0$ is the {\bf Bessel function of the first kind} of order $0$. The random field $B_E$ was introduced by Berry \cite{Be:77,Be:02} and subsequently studied in several works  --- see for instance \cite{NPR:19, KW:18, BCW:19,CH:16, PV:20, DNPR:19}. 

\medskip

To simplify the discussion, we will sometimes write $b = \{b(x) : x\in \R^2\}$ to indicate the random field $B_E$ for $E = (4\pi^2)^{-1}$, in such a way that $b$ is almost surely a Laplace eigenfunction with unit eigenvalue. One should regard the random field $b$ as a {\it canonical} Gaussian Laplace eigenfunction on $\R^2$, emerging as a local scaling limit in a number of models of Gaussian random fields on 2-manifolds --- see e.g. \cite{CH:16, DNPR:19, MPRW:16, KKW:13, Wi:10, W:22, Ze:09} for explicit examples. We also recall that, in resonance with {\bf Berry's conjecture} (see e.g. \cite{ABLM:21, CH:16, I:21, W:22}, as well as \cite{Be:77} for the original formulation), the field $b$ is believed to be a model for the high-energy behaviour of deterministic Laplace eigenfunctions on manifolds where classical dynamics are sufficiently chaotic. \\

\noindent \textbf{Some notation.} Given two sequences of positive numbers $\lbrace a_n\rbrace_n$ and $\lbrace b_n\rbrace_n$, we write $a_n=o(b_n)$ if $a_n/b_n \to 0$ as $n\to +\infty$, $a_n = O(b_n)$ if $a_n/b_n$ is asymptotically bounded, $a_n \sim b_n$ if $a_n/b_n \to 1$ as $n\to +\infty$ and $a_n \approx b_n$ if $a_n/b_n \to c$ as $n\to +\infty$, where $c$ is a constant that does not depend on $n$. Moreover, for random variables $\lbrace X_n\rbrace_n$ and $X$ we write $X_n \stackrel{d}{\to} X$ if the sequence $X_n$ converges to $X$ in distribution.

\subsection{Fluctuations of nodal sets}

As anticipated, in this paper we will focus on the {\bf high-energy behaviour} of the {\bf nodal set} $$B_E^{-1}(0) = \{x\in \R^2 : B_E(x) = 0\}.$$  It is a well-known fact \cite{CH:16, NPR:19, PV:20} that $B_E^{-1}(0)$ is almost surely equal to the union of disjoint rectifiable curves (called {\bf nodal lines}). Local and non-local functionals of such a random set have been recently the object of an intense study, especially in connection with nodal statistics of (approximate) Laplace eigenfunctions on Riemann surfaces (see \cite{CH:16, DNPR:19, KW:18, MRVKS:21, NS:16, W:22} for a sample of recent contributions). 

\medskip

For every Borel set $\mathcal{D}\subseteq \R^2$, now write
\begin{equation}\label{e:length}
\mathscr{L}_{E}(\mathcal{D}):=\mathcal{H}^1(B_E^{-1}(0)\cap \mathcal{D}),
\end{equation}
where $\mathcal{H}^1$ denotes the $1$-dimensional Hausdorff measure, to indicate the length of the portion of the nodal lines of $B_E$ contained in $\mathcal{D}$. From the above discussion, it is clear that the mapping $\mathcal{D} \mapsto \mathscr{L}_{E}(\mathcal{D})$ defines with probability one a set function with the following elementary properties: 
\begin{enumerate}

\item $\mathscr{L}_{E}$ is a locally finite measure on $\mathscr{B}(\R^2)$;
\item The support of $\mathscr{L}_{E}$ is $B_E^{-1}(0)$;
\item The restriction of $\mathscr{L}_{E}$ to any square $ \mathcal{Q} = [a,b]\times [c,d]$ is completely determined by the random partition function 
\begin{equation}\label{e:rf}
(t_1, t_2)\mapsto \mathscr{L}_E([a,t_1] \times [c,t_2]), \quad (t_1,t_2)\in \mathcal{Q}.
\end{equation}
\end{enumerate}

\medskip

Several results characterizing the high-energy behaviour of $\mathscr{L}_E$ are available. The most relevant for our study are reported in the next statement. From now on, we use the symbol $\mathscr{A}$ to denote the collection of all piecewise $C^1$ simply connected compact subsets of $\R^2$ having non-empty interior.

\begin{Thm} \label{thm_L}

\begin{enumerate}

\item[\rm \bf 1.] {\rm (See \cite{Be:02, NPR:19})} For every $\D\in \A$, one has that

 \begin{eqnarray*}
\E{\mathscr{L}_E(\D)} = \mathrm{area}(\D) \frac{\pi}{\sqrt{2}}\sqrt{E} \quad \mbox{and}\quad
\V{\mathscr{L}_E(\D)} \sim \mathrm{area}(\D) \frac{\log E}{512\pi},
\end{eqnarray*}
as $E\to \infty$.

\item[\rm\bf 2.] {\rm (See Theorem 1.1 in \cite{NPR:19})} For every $\D\in \A$, as $E\to \infty$,
\begin{eqnarray}\label{CLT1}
\wt{\mathscr{L}}_E(\D):= \sqrt{\frac{512\pi}{\log E}}\big(\mathscr{L}_E(\D)-\E{\mathscr{L}_E(\D)}\big)
\Law \mathcal{N}(0,{\rm area}(\mathcal{D}) ),
\end{eqnarray}
where $\mathcal{N}(\mu,\sigma^2)$ denotes the one-dimensional Gaussian distribution with mean $\mu$ and variance $\sigma^2$.

\item[\rm\bf 3.] {\rm (See Theorem 3.2 in \cite{PV:20})} For every integer $d\geq 1$ and every fixed $\D_1,\D_2,\dots,\D_d\in\A$, we define the $d\times d$ matrix $\Sigma = \{\Sigma(i,j):i,j=1,\ldots,d \}$ by the relation
\begin{equation}\label{lim_cov_ij}
\Sigma(i,j):=\mathrm{area}(\D_i\cap \D_j)\,.
\end{equation}
Then, as $E\to \infty$, one has that
\begin{equation}\label{e:mainlength}
\(\widetilde{\mathscr{L}}_{E}(\D_1),\widetilde{\mathscr{L}}_{E}(\D_2),\dots,\widetilde{\mathscr{L}}_{E}(\D_d)\) \Law \mathcal{N}_d(0,\Sigma), 
\end{equation}
where $ \mathcal{N}_d(0,\Sigma)$ stands for the centered  $d$-dimensional  Gaussian distribution with covariance $\Sigma$.

\end{enumerate}
\end{Thm}

\medskip

Theorem \ref{thm_L}--${\bf 3.}$  shows that, in the high-energy limit, 
the finite-dimensional distributions of the process $\{\wt{\mathscr{L}}_E(\D):\D\in \A \}$ converge to those of a Gaussian field with the covariance structure of a homogeneous independently scattered random measure with unit intensity. Such a characterization immediately implies the following result. 
\begin{Prop}[See \cite{PV:20}]\label{p:fdd}

Define the random field $X_E=\{X_E(t_1,t_2):(t_1,t_2)\in[0,1]^2\}$ as 
\begin{equation}\label{Eq:XE}
  X_E(t_1,t_2):= \wt{\mathscr{L}}_{E}([0,t_1]\times [0,t_2]).
\end{equation}
Then, as $E\to \infty$, $X_E$ converges in the sense of finite-dimensional distributions  to a standard  \textit{Wiener sheet}, that is, to 
a centered Gaussian process $\mathbf{W}=\{\mathbf{W}(t_1,t_2):(t_1,t_2)\in [0,1]^2\}$ with covariance function $\E{\mathbf{W}(t_1,t_2)\mathbf{W}(s_1,s_2)}= (t_1\wedge s_1)(t_2\wedge s_2)$. 
\end{Prop}

We note that the choice of the specific rectangle $[0,1]^2$ is immaterial: in particular, the content of Proposition \ref{p:fdd} can be easily adapted to deal with arbitrary regions $\mathcal{Q} = [a,b]\times [c,d]$. It is a remarkable fact -- not noted in \cite{PV:20} -- that the content of Proposition \ref{p:fdd} allows one to directly conclude that the signed measure $\wt{\mathscr{L}}_E$ defined in \eqref{CLT1} converges to a standard white noise on the space of generalised functions $\mathscr{D}'(R)$, with $R = (0,1)^2$. Such an implication is made clear in the next statement, where we use the symbol $ \mathcal{C}_c^\infty(R)$ to indicate the class of compactly supported smooth mappings on $R$; we refer the reader to \cite{F:67} for definitions and background.

\begin{Prop}[Convergence to white noise]\label{p:whitenoise} Let the above notation prevail. For every $\varphi \in \mathcal{C}_c^\infty(R)$, define
$$
\langle \wt{\mathscr{L}}_E, \varphi \rangle := \sqrt{\frac{512\pi}{\log E}}\left(\int_R \varphi({\bf t}) \mathscr{L}_E(d{\bf t}) - \frac{\pi}{\sqrt{2}}\sqrt{E} \int_R\varphi({\bf t}) d{\bf t} \right).
$$
Then, the mapping $ \varphi \mapsto \langle \wt{\mathscr{L}}_E, \varphi \rangle$ is a random element with values in $\mathscr{D}'(R)$ and, as $E\to \infty$, $\langle \wt{\mathscr{L}}_E, \bullet \rangle$ converges in distribution, in the sense of $\mathscr{D}'(R)$, to a standard white noise. In particular, for every integer $m\geq1$ and every $\varphi_1,..., \varphi_m\in\mathcal{C}_c^\infty(R)$, the vector $$\(\langle \wt{\mathscr{L}}_E, \varphi_1\rangle,\dots, \langle \widetilde{\mathscr{L}}_{E}, \varphi_m \rangle\)$$ converges in distribution to a $m$-dimensional centered Gaussian vector with covariance
$$
\Lambda(i,j) = \int_R \varphi_i({\bf t}) \varphi_j({\bf t}) d{\bf t}, \quad i,j=1,...,m.
$$
\end{Prop}

For the sake of completeness, the proof of Proposition \ref{p:whitenoise} is reported in Section \ref{ss:proofwhitenoise}. See e.g. \cite{Wi:10, AL:21} for similar results involving, respectively, the nodal set of random spherical harmonics and the roots of Kostlan polynomials.

\medskip

\noindent{\bf Remark on notation}. From now on, we will freely use the language of Gaussian stochastic analysis and {\bf Wiener chaos expansions}, as detailed e.g. in \cite[Chapter 2]{NP:12} or \cite[Chapter 1]{N:06}. In particular, given a square-integrable random variable $X$ that is measurable with respect to the $\sigma$-field $\sigma(G)$ generated by a separable Gaussian field $G$, we write $X[q]$, $q=0,1,\ldots,$ to indicate the projection of $X$ onto the $q$th Wiener chaos associated with $G$, in such a way that $X = \sum_{q\geq 0} X[q]$, where the series converges in $L^2$. Similarly, given a square-integrable and $\sigma(G)$-measurable random field $Z = \{ Z({ \bf t}) : {\bf t} = (t_1,t_2)\in [0,1]^2\}$, we will write $Z[q] :=  \{ Z[q]({ \bf t} ) :  { \bf t} \in [0,1]^2\}$, where 
$Z[q]({ \bf t})$ is the projection of $Z({ \bf t})$ onto the $q$th Wiener chaos. Applying these conventions to the normalized nodal length process $Z =X_E$ introduced in \eqref{Eq:XE}, one obtains the Wiener-It\^{o} representation
\begin{eqnarray}\label{Eq:XE24R}
X_E = X_E[2]+ X_E[4] + R_E, \quad  {\rm where}\,\,
R_E({\bf t}):= \sum_{q\geq 3} X_E[2q]({\bf t}),
\end{eqnarray}
and the series converges in $L^2$ for every fixed ${\bf t}$.

%
%
%
%

\subsection{The main question}\label{ss:mainquestion}

Some additional notation is required in order to frame our contribution. Consider the unit square $[0,1]^2$ and, for all fixed $\bb{t}=(t_1,t_2)\in [0,1]^2$, define the following four regions: 
\begin{eqnarray*}
Q(\bb{t},NE) &:=& \{ \bb{s}=(s_1,s_2)\in [0,1]^2 : s_1>t_1, s_2>t_2\} \\
Q(\bb{t},NW) &:=& \{ \bb{s}=(s_1,s_2)\in [0,1]^2 : s_1<t_1, s_2>t_2\} \\
Q(\bb{t},SW) &:= &\{ \bb{s}=(s_1,s_2)\in [0,1]^2 : s_1<t_1, s_2<t_2\}\\
Q(\bb{t},SE) &:=& \{ \bb{s}=(s_1,s_2)\in [0,1]^2 : s_1>t_1, s_2<t_2\}. 
\end{eqnarray*}
We remark that some of these regions may be empty, in the case where $\bb{t}$ belongs to the boundary of $[0,1]^2$.
The {\bf Skorohod space} $\mathbf{D}_2=D([0,1]^2,\R)$ is the class of functions $f:[0,1]^2 \to \R$ verifying the following continuity property for every $\bb{t}\in [0,1]^2$: for every $R \in \{NE,NW,SW,SE\}$ and every sequence $\{\bb{t}_n:n\geq 1\} \subset Q(\bb{t},R)$ such that $\bb{t}_n \to \bb{t}$ as $n\to \infty$, the limit $\lim_{n\to\infty} f(\bb{t}_n)$ exists and is finite, and, moreover, for every sequence $\{\bb{t}_n:n\geq 1\}\subset Q(\bb{t}, NE)$ such that $\bb{t}_n \to \bb{t}$ as $n\to \infty$, one has that 
$\lim_{n\to \infty} f(\bb{t}_n)= f(\bb{t})$.

\medskip
We endow the space $\mathbf{D}_2$  both with the $\sigma$-field generated by coordinate projections, and with the
Skorohod topology described in Neuhaus \cite[p.1289]{Ne:71}; such a topology is generated by a distance, noted $d$ in \cite{Ne:71}, making $\mathbf{D}_2$ a separable metric space. We also define $\mathbf{C}_2=C([0,1]^2,\R)$ to be the subset
of $\mathbf{D}_2$ composed of continuous mappings. 

\medskip

We note that, in the case where $f:[0,1]^2\to \R$ takes the form $f(\bb{t}):=\mu([0,t_1]\times [0,t_2])$ for    some    finite measure  $\mu$ on $[0,1]^2$, it follows from an  application of the dominated convergence theorem  that $f \in \mathbf{D}_2$. 
In view of the above discussion, the nodal length processes $\{X_E:E>0\}$ defined in \eqref{Eq:XE} are thus $\mathbf{D}_2$-valued random elements. 

\medskip

Our aim in this paper is to initiate the study of the following question, that was left open in \cite{PV:20}.

\bigskip

\noindent{\bf Question A}. {\it As $E\to \infty$, does $X_E$ converge in distribution to a standard Brownian sheet ${\bf W}$ in the Skorohod space ${\bf D}_2$?
}
\medskip

As explained in \cite[Section 4.3]{PV:20} and in the forthcoming Section \ref{ss:discretized}, an affirmative answer to {\bf Question A} would yield a convergence result that is strictly stronger than the convergence in the senses of finite-dimensional distributions and of generalized functions, featured in Propositions \ref{p:fdd} and \ref{p:whitenoise} above --- implying in particular limit theorems for random variables depending on the maximum of $X_E$ or of its absolute value. In view of the results of \cite{DNPR:22}, these result would immediately extend to monochromatic random waves on Riemann surfaces without conjugate points.

\medskip

We choose to base our analysis of {\bf Question A} on the following standard lemma, whose proof is deferred to Appendix \ref{App:UVW}.
\begin{Lem}\label{Lem:UVW}
Let $\{X, X_n :n\geq 1 \}$ be a collection of random fields with values in $\mathbf{D}_2$ such that $\Prob(X\in \mathbf{C}_2) = 1$. We assume that, for every $n\geq 1$, the process 
$X_n$ can be written as $X_n = U_n+V_n+W_n$, where the fields $U_n,V_n, W_n$ are such that
\begin{enumerate}[label=(\roman*)]
\item[\rm (a)]  as $n\to \infty$, $U_n$ converges weakly to $X$ in   $\mathbf{D}_2$,
\item[\rm (b)] as $n\to \infty$, $V_n$ converges weakly to zero in   $\mathbf{D}_2$,
\item[\rm (c)] for every $\eps>0$,
\begin{eqnarray*}
\lim_{n\to \infty} \Prob\{ \sup_{\bb{t}\in [0,1]^2} \left|W_n(\bb{t})\right|>\eps\} = 0,
\end{eqnarray*}
\end{enumerate}
Then, $X_n$ converges weakly to $X$ in $\mathbf{D}_2$.
\end{Lem}
In particular, applying Lemma \ref{Lem:UVW} to $(X_n,U_n,V_n,W_n) = (X_E,X_E[4], X_E[2], R_E)$ in \eqref{Eq:XE24R} (and noting that $\Prob\{ X_E[2] \in \mathbf{C}_2\}=\Prob\{ X_E[4] \in \mathbf{C}_2\}=1$, for every $E>0)$, suggests the following three-step strategy for positively answering {\bf Question A}:
\begin{enumerate}[label=\textbf{(\Roman*)}]
\item  prove  that the projection $X_E[4]$ converges weakly to a standard  Wiener sheet as $E\to \infty$;
\item prove that the second chaotic projection $X_E[2]$ associated with $X_E$ converges weakly to zero, as $E\to \infty$;
\item prove that the residual term $R_E$ converges uniformly to zero in probability, as $E\to \infty$. 
\end{enumerate}

\medskip

For ease of reference, we will refer to {\bf (I)} -- {\bf (III)} as ``{\bf Strategy S}''. The next result, proved in \cite[Theorem 3.4]{PV:20}, settles Point {\bf (I)}. 
\begin{Thm}[\cite{PV:20}, Theorem 3.4]\label{Thm:DomX4}
For every fixed ${\bf t} \in [0,1]^2$, one has that, as $E\to \infty$,
\begin{eqnarray*}
\E{\( X_E({\bf t}) - X_E[4]({\bf t})\)^2} \to 0.
\end{eqnarray*}
Moreover, as $E\to \infty$, $X_E[4]$ converges weakly to $\mathbf{W}$ in the Skorohod space $\mathbf{D}_2$. 
\end{Thm}

\medskip

As explained in the forthcoming Section \ref{s:main}, the principal aim of the present work is to fully address Point {\bf (II)}, as well as to provide some decisive progress towards a full achievement of Point {\bf (III)}. We will see in Section \ref{ss:second} that our way of attacking Point {\bf (II)} reveals an intriguing connection with the CLTs for zeros of Gaussian entire functions established in \cite{BS:17, ST:04}. On the other hand, our analysis of Point {\bf (III)} will allow one to prove functional convergence results for some {\bf discretized versions} of the nodal fields $\{ X_E : E>0\}$ --- see Section \ref{ss:approximate}.

%
%
%

\section{Main results}\label{s:main}

\subsection{Second chaos and total disorder}\label{ss:skhoro}

Let $\D$ be a planar domain with piecewise $C^1$ boundary $\partial \D$. 
In \cite[Lemma 4.1]{NPR:19}, the authors prove that the projection on the second Wiener chaos of the nodal length $\mathscr{L}_E(\mathcal{D})$, as defined in \eqref{e:length}, is given by
\begin{eqnarray}\label{Second}
\mathscr{L}_E[2](\mathcal{D})= \frac{1}{8\pi \sqrt{2E}} \int_{\partial\mathcal{D}} B_E(x) \scal{\nabla B_E(x)}{\n_{ \mathcal{ D}}(x)} \mathcal{H}^1(dx),
\end{eqnarray}
where $\n_{ \mathcal{D}}(x)=(\n^1_{ \mathcal{D}}(x),\n^2_{ \mathcal{D}}(x))$ is the outward unit normal vector to $\partial\mathcal{D}$ at $x$. One of the main contributions of our work is a full characterization of the joint fluctuations of the random variables $\mathscr{L}_E[2](\mathcal{D})$, as $E\to\infty$, whenever $\mathcal{D}$ is a polygonal domain. 

\subsubsection{Some random fields indexed by curves}

We will actually study the fluctuations of \eqref{Second} in the context of slightly more general random objects. 

\begin{Def}\label{d:chains}\begin{itemize}
\item[(a)] An {\bf oriented segment} $S$ is the image of a mapping 
\begin{align*}
\gamma &: [0, L] \to \R^2 \\
&: \quad t \,\,\,\,\,\mapsto p + t (\cos\theta, \sin\theta),
\end{align*}
where $L>0$ is the {\bf length} of $S$, $p\in \R^2$ and $\theta \in [0, 2\pi )$. 

\item[(b)] A (simple) {\bf polygonal chain} is an ordered collection $\mathcal{C} = (S_1,...,S_m)$ of oriented segments such that $\gamma_k (L_k) = \gamma_{k+1}(0) = p_{k+1}$ (with obvious notation), for all $k=1,...,m-1$, and the union $\cup S_i$ defines a simple curve in $\R^2$; we will say that $\mathcal{C}$ is {\bf closed} if $\gamma_m(L_m) = \gamma_1(0)$. The class of all polygonal chains is denoted by $\mathscr{C}$. In the discussion below, we will often identify a chain $\mathcal C$ with the oriented curve defined by the union $\cup S_i$. 

\item[(c)] Given $\mathcal{C} = (S_1,...,S_m) \in \mathscr{C}$ and $x$ belonging to the relative interior of $S_k$, $k=1,...,m$, we define $\n_{\mathcal C} (x) = (-\sin \theta_k, \cos\theta_k)$ to be the {\bf unit normal vector} to $\mathcal{C}$ at the point $x$. The definition of $\n_{\mathcal C} (x) $ is arbitrarily extended (for instance, by setting $\n_{\mathcal C} (x) =0$) to the remaining (finitely many) points $x$ in $\mathcal{C}$. If the chain $\mathcal{C}$ coincides with a single segment $S$, then $\n_{\mathcal C} (x) = \n_S(x)$ is independent of the choice of $x$ and we simply write $\n_S$ to indicate the normal vector common to all elements of the relative interior of $S$.

\item[(d)] Given two oriented segments $S_1, S_2$ we define the {\bf signed length} of the intersection $S_1\cap S_2$ as $\lambda (S_1, S_2) := {\rm length} (S_1\cap S_2) \,  \langle \n_{S_1}, \n_{S_2} \rangle$, where $\langle \cdot, \cdot \rangle$ stands for the usual Euclidean inner product; plainly, if ${\rm length} (S_1\cap S_2)>0$, then $\langle \n_{S_1}, \n_{S_2}\rangle$ equals either 1 or $-1$, depending on whether $S_1$ and $S_2$ have the same, or opposite, orientations. We extend the definition of $\lambda$ to all pairs $\mathcal{A} = (T_1,...,T_n), \,  \mathcal{C} = (S_1,....,S_m)\in \mathscr{C}$ by setting
\begin{equation}\label{signed}
\lambda(\mathcal{A} , \mathcal{C}) := \sum_{i=1}^n \sum_{k=1}^m \lambda (T_i, S_k). 
\end{equation}
It is easily checked that
\begin{equation}\label{signed2}
\lambda(\mathcal{A} , \mathcal{C}) = \int_{\mathcal{ C} \cap \mathcal {A}}  \langle \n_{\mathcal{C}}(x),\n_{\mathcal{A}}(x) \rangle \,  \mathcal{H}^1(dx).   
\end{equation}
The reader is referred to \cite[Definition 3]{BS:17}, and the discussion therein, for more properties of signed lengths. 
\end{itemize}
\end{Def}

\medskip

For $\mathcal{C} = (S_1,..., S_m) \in \mathscr{C}$ and $E>0$, we now define the random variable
\begin{eqnarray}\label{Phi}
\phi_E(\mathcal{C})&:=& \frac{1}{8\pi \sqrt{2E}} \int_{\mathcal{C}} B_E(x) \scal{\nabla B_E(x)}{\n_{\mathcal{C}}(x)} \mathcal{H}^1(dx) \\
&=&  \frac{1}{8\pi \sqrt{2E}}  \sum_{k=1}^m \int_0^{L_k } B_E(\gamma_k(t) ) \langle \nabla B_E(\gamma_k(t)), \n_{S_k} \rangle dt,\label{Phi2}
\end{eqnarray}
where \eqref{Phi2} is a straightforward representation of \eqref{Phi} in terms of line integrals.

\begin{Rem}\label{r:chains}
\begin{itemize}

\item[(a)] When $\mathcal{C} \in \mathscr{C}$ is clockwise oriented and closed, then \cite[Lemma 4.1]{NPR:19} implies that $\phi_E(\mathcal{C}) = \mathscr{L}_E[2](\mathcal{D}),$ where $\mathcal{D}$ is the polygonal domain enclosed by $\mathcal{C}$ (if $\mathcal{C}$ is counterclockwise oriented, then the equality continues to hold, but with a minus sign in front of the right-hand side). {\it From now on, we will conventionally assume that the boundary $\partial\mathcal{D}$ of any polygonal region $\mathcal{D}$ is a clockwise oriented closed chain}.

\item[(b)] Let $\mathcal D$ be a polygonal domain, fix $R>0$, and select $E>0$ in such a way that $R = 2\pi\sqrt{E}$. Then, a direct computation (based e.g. on the arguments developed in \cite[Proof of Lemma 4.1]{NPR:19}) shows that 
\begin{equation}\label{e:rescaling}
\phi_E(\partial\mathcal{D}) =\mathscr{L}_E[2](\mathcal{D})  \eqLaw R^{-1} \mathscr{L}(b; R\cdot\mathcal{D})[2],
\end{equation}
where $\mathscr{L}(b; R \cdot \mathcal{D})$ denotes the nodal length of the field $b$ (as defined in Section \ref{ss:model}), restricted to the region $R\cdot \mathcal{D}$.

\item[(c)]  In general, it is easy to see that, for every $\CC\in \mathscr{C}$, the random variable $\phi_E(\mathcal{C})$ is an element of the second Wiener chaos associated with $B_E$. Indeed, denoting by $I_p$ the Wiener isometry of order $p$, and writing $B_E(x)=I_1(f^E_0(x,\cdot)), \partial_jB_E(x)=\sqrt{2\pi^2E} I_1(f^E_j(x,\cdot)), j=1,2$, for suitable  kernels $f^E_j(x,\cdot), j=0,1,2$, defined e.g. on the Hilbert space $L^2([0,1], dx)$ (with $dx$ denoting the Lebesgue measure), 
one can show that 
$\phi_E(\mathcal{C})=I_2(u^E(\CC))$, where 
\begin{eqnarray}\label{kernel}
u^E(\CC) = \frac{1}{8} \sum_{j=1}^2 \int_{\CC} f_0^E(x,\cdot) \widetilde{\otimes}
f^E_j(x,\cdot) \n^j_{\CC}(x) \mathcal{H}^1(dx),
\end{eqnarray}
 where $\widetilde{\otimes}$ denotes the canonical symmetrization of the tensor product, and $\n_{\mathcal{C}}(x)=(\n_{\mathcal{C}}^1(x), \n_{\mathcal{C}}^2(x))$. We refer the reader to the proof of Proposition \ref{PropCLTLS} for more details, and to \cite[Chapter 2]{NP:12} for background on Wiener chaos.

\item[(d)] It is of course possible to extend the above definitions to the case where $\mathcal{C}$ is a piecewise smooth simple curve of finite length (or even to the case where $\mathcal{C}$ is a $\R$-{\it chain}, in the sense of \cite[Definition 1]{BS:17}). Since our techniques only allow one to deal with the case of polygonal chains, we decided not to pursue such a level of generality. 

\end{itemize}

\end{Rem}

\subsubsection{Second order results}\label{ss:second}

We now state our main results concerning the random variables 
$\phi_E(\CC)$. The first statement characterizes their asymptotic 
covariance structure in the high-energy regime.
\begin{Thm}[Asymptotic covariance structure]\label{AsyVar} For every $\CC_1,\CC_2\in\mathscr{C}$, as $E\to \infty$, 
\begin{eqnarray}\label{Eq:AsyCov}
\Cov{\phi_E(\CC_1)}{\phi_E(\CC_2)} 
= \frac{\lambda(\CC_1,\CC_2)}{16\pi^2\sqrt{E}} + o\(\frac{1}{\sqrt{E}}\),
\end{eqnarray}
where $\lambda(\CC_1,\CC_2)$ indicates the signed length of $\CC_1\cap\CC_2$, as introduced in Definition \ref{d:chains}-{\rm (d)}. In particular, as $E\to \infty$,  
\begin{eqnarray}\label{Eq:AsyVar}
\V{\phi_E(  \CC)} 
= \frac{{\rm length} (\CC)}{16\pi^2\sqrt{E}} + o\left(\frac{1}{\sqrt{E}}\right).
\end{eqnarray}
\end{Thm}

\smallskip

Specializing the content of Theorem \ref{AsyVar} to the case where $\CC=\partial\D$ is the boundary of a polygonal domain $\mathcal{D}\subset\R^2$ and bearing in mind Remark \ref{r:chains}-(a), one infers that
\begin{eqnarray*}
\V{\mathscr{L}_E[2](   \D)} 
= \frac{{\rm length} (\partial\D)}{16\pi^2\sqrt{E}} + o\left(\frac{1}{\sqrt{E}}\right),
\end{eqnarray*}
as $E\to\infty$. 
This estimate refines 
the upper bound $O(1)$ for the variance of  $\mathscr{L}_E[2](   \D)$ derived in \cite[Lemma 4.1]{NPR:19} and shows in particular that $\mathscr{L}_E[2](\mathcal{D})$ vanishes in the high-energy limit.

\medskip
In view of \eqref{Eq:AsyVar}, we introduce the class of normalized random variables: 
\begin{eqnarray}\label{normal}
\wt{\phi}_{E}(  \CC) := 4\pi E^{1/4} \phi_E(  \CC), \quad \CC\in \mathscr{C}.
\end{eqnarray}
The following statement establishes a multidimensional convergence result -- in the sense of finite-dimensional distributions -- for the random field
$\{ \wt{\phi}_{E}(  \CC): \CC\in \mathscr{C}\}$. 
\begin{Thm}\label{CLTd}
For every integer $d\geq 1$ and every  \allowbreak $\CC_1,\ldots, \CC_d \in \mathscr{C}$, we have that, as $E\to \infty$, 
\begin{eqnarray*}
\left( 
\wt{\phi}_{E}(\CC_1), \ldots, \wt{\phi}_{E}(\CC_d)\right) 
\Law \mathcal{N}_d(0, \Sigma),
\end{eqnarray*}
where $\Sigma=\{\Sigma(i,j):i,j=1,\ldots,d\}$ is the $d\times d$ matrix defined by
\begin{eqnarray}\label{e:sigmalength}
\Sigma(i,j) :=  \lambda(\CC_i,\CC_j) \ , \quad i,j=1,\ldots, d.
\end{eqnarray}
\end{Thm}

Theorem \ref{CLTd} shows that, in the high-frequency regime, the covariance of $\wt{\phi}_{E}(\CC_1)$ and $\wt{\phi}_{E}(\CC_2)$ depends on the geometry of   $\CC_1$ and $\CC_2$ through the signed length of their intersection. In particular, when $\lambda(\CC_1 , \CC_2)$ is zero (which is the case when $\CC_1$ and $\CC_2$ intersect in finitely many points or are disjoint), $\wt{\phi}_{E}(\CC_1)$ and $\wt{\phi}_{E}(\CC_2)$ converge to independent centered Gaussian random variables with variances ${\rm length} (\CC_1)$ and ${\rm length}(\CC_2)$, respectively. 

\smallskip

Using the content of Remark \ref{r:chains}-(b), one can immediately deduce a joint CLT for the second chaos projections of Berry's nodal lengths on expanding domains.

\begin{Prop}\label{p:largedomain} Let $\mathcal{D}_1,..., \mathcal{D}_d$ be polygonal domains. Then, as $R\to \infty$,
\begin{equation}\label{e:scaledclt}
\sqrt{\frac{8\pi}{R} } \Big(\mathscr{L}(b; R\cdot\mathcal{D}_1)[2],...,\mathscr{L}(b; R \cdot\mathcal{D}_d)[2]    \Big)     \Law \mathcal{N}_d(0, \Sigma)
\end{equation}
where $\Sigma$ is given by \eqref{e:sigmalength} with $\CC_i = \partial\mathcal{D}_i$. In particular, one has that, for a fixed polygonal domain $\mathcal {D}$,
\begin{equation}\label{e:hu}
{\bf Var} \big(\mathscr{L}(b; R\cdot \mathcal{D})[2]\big) = \frac{ R \cdot {\rm length} (  \partial \mathcal{D}  )}{8\pi} + o(R), \quad R\to \infty. 
\end{equation}
\end{Prop}

\begin{Rem}[Hyperuniform scalings and total disorder] 
\begin{itemize}

\item[(a)] Theorem \ref{CLTd} and Proposition \ref{p:largedomain} should be compared with the content of \cite[Theorem 1]{BS:17} (see also \cite{ST:04}). In such a reference, the authors consider the Gaussian entire function
$$
z\mapsto f(z) = \sum_{n=0}^\infty \zeta_n \frac{z^n}{n!}, \quad z\in \mathbb{C},
$$
where $\{\zeta_n : n\geq 0\}$ is a sequence of i.i.d. standard complex Gaussian random variables, and study the joint fluctuations of random variables of the type $\Delta_R(\mathcal{C})$, where $\mathcal{C}$ indicates a smooth simple curve of finite length (not necesarily polygonal) and, for $R>0$, $\Delta_R(\mathcal{C})$ represents the {\bf increment of the argument} of $f(Rz)$ along $\mathcal{C}$. The main finding of \cite{BS:17} is that, for the constant $c = \sqrt{\pi/4}\,  \zeta(3/2)$,
\begin{equation}\label{e:v1}
{\bf Var} (\Delta_R(\mathcal{C}) )= c R\cdot  {\rm length} (\mathcal{C})  + o(R), \quad R\to\infty,
\end{equation}
and 
\begin{equation}\label{e:clt1}
\frac{1}{\sqrt{cR}} \left( (\Delta_R(\mathcal{C}_1) - \mathbb{E}(\Delta_R(\mathcal{C}_1)),...,(\Delta_R(\mathcal{C}_d) - \mathbb{E}\Delta_R(\mathcal{C}_d))\right) \Law \mathcal{N}_d(0, \Sigma),
\end{equation}
where $\Sigma$ is given by \eqref{e:sigmalength} (once the notion of signed length is extended to generic smooth curves by using \eqref{signed2}). Specializing these findings to the case where each $\mathcal{C}_i$ is the counterclockwise oriented boundary of a smooth domain $\mathcal{D}_i$, and denoting by $n_R(\mathcal{D})$ the number of zeros of $f$ in the domain $R \cdot\mathcal{D}$, one deduces that, for $c_0 = \zeta(3/2) / (8 \pi^{3/2})$,
\begin{equation}\label{e:v2}
{\bf Var} \left(n_R(\mathcal{D}) )= c_0 R\cdot  {\rm length} ( \partial \mathcal{D} \right)  + o(R), \quad R\to\infty,
\end{equation}
and 
\begin{equation}\label{e:clt2}
\frac{1}{\sqrt{c_0R}} \left( (n_R(\mathcal{D}_1) - \mathbb{E}(n_R(\mathcal{D}_1)),...,(n_R(\mathcal{D}_d) - \mathbb{E}\n_R(\mathcal{D}_d))\right) \Law \mathcal{N}_d(0, \Sigma),
\end{equation}
where $\Sigma$ is given by \eqref{e:sigmalength} with $\CC_i = \partial\mathcal{D}_i$. This yields joint Gaussian fluctuations of the same nature as \eqref{e:scaledclt}.

\item[(b)] Variance estimates such as \eqref{e:hu} and \eqref{e:v2} mean that, as $R\to\infty$, the variances of $\mathscr{L}(b; R\cdot \mathcal{D})$ and $n_R(\mathcal{D})$ scale as the length of the boundary of $R\cdot \mathcal{D}$, rather than as ${\rm area} (R\cdot \mathcal{D}) \asymp R^2$. Such a behaviour emerges for many point processes that are relevant in modern probability and statistical mechanics, and is known as {\bf hyperuniformity} --- see Point (d) below, as well as \cite{GL:17, T:17} and the references therein. 

\item[(c)] Consider a centered Gaussian field $G = \{G(\mathcal{C}) : \mathcal{C} \in \mathscr{C}\}$ indexed by polygonal chains and such that $\mathbb{E}[G(\CC_1) G(\CC_2)] = \lambda(\CC_1, \CC_2)$. Then, it is easily seen that $G$ is necessarily a {\bf total disorder process}, that is, the linear span of $G$ contains an {\it uncountable} collection of i.i.d. centered Gaussian random variables with unit variance. The same conclusion holds if one replaces $\mathscr{C}$ with the collection of closed polygonal chains, or even with the set of closed chains that are the boundary of a rectangle. 

\item[(d)] A multivariate CLT similar to \eqref{e:scaledclt} and \eqref{e:clt2} appears in the physics paper \cite{Le:83} on charge fluctuations for Coulomb systems. Here, the author considers the net electric charge $Q_{\Lambda}$ contained in a subregion $\Lambda$ of an infinite equilibrium system and establishes joint CLTs for pairs $(Q_{\Lambda_1}, Q_{\Lambda_2})$, where $\Lambda_1,\Lambda_2$ are growing regions. For instance, for cubes $\Lambda_1, \Lambda_2$ of side length $L\to \infty$, it turns out that variance of the considered objects scales hyperuniformly as $\asymp L$, and the limiting covariance is only non-zero when the cubes share a face. Total disorder processes also appear in several works in random matrix theory. In \cite{Wi:02} (see also \cite[Theorem 6.3]{DE:01}) the authors consider the number $N_n(\alpha,\beta)$ of eigenvalues lying in a circular interval $(e^{i\alpha},e^{i\beta})$ of random $n\times n$ unitary matrices sampled according to the Haar measure. It is shown that the finite-dimensional distributions of the normalized process 
\begin{eqnarray*}
\{ \frac{N_n(\alpha,\beta)-\E{N_n(\alpha,\beta)}}{\pi^{-1}\sqrt{\log n}}: 0<\alpha<\beta<2\pi \}
\end{eqnarray*}
converge to those of a centred Gaussian process $\{Z(\alpha,\beta): 0<\alpha<\beta<2\pi\}$ with covariance function
\begin{eqnarray*}
\E{Z(\alpha,\beta) Z(\alpha',\beta')}
= 
\{
\begin{array}{ll}
1 &\alpha=\alpha', \beta=\beta'\\
-1 & \alpha=\beta', \alpha'=\beta\\
1/2 & \alpha=\alpha' \mathrm{ \ or \ } \beta=\beta' \mathrm{ \ but \ not \ both}\\
-1/2 & \alpha=\beta' \mathrm{ \ or \ } \beta=\alpha' \mathrm{ \ but \ not \ both}\\
0 & \alpha,\beta,\alpha',\beta' \mathrm{ \ distinct}
\end{array}
\right. .
\end{eqnarray*}
From such a covariance structure, it becomes clear that, unless two intervals $(e^{i\alpha},e^{i\beta})$ and $(e^{i\alpha'},e^{i\beta'})$ have at least one endpoint in common, the limiting random variables $Z(\alpha,\beta)$ and $ Z(\alpha',\beta')$ are independent. Finally, in \cite{HNY:08} the authors prove that the finite-dimensional distributions of a complex Gaussian total disorder process appear as the limiting distribution of the multi-dimensional extension of Selberg's Central Limit Theorem for the logarithm of the Riemann Zeta function (see \cite{Se:46,Se:92}).

\end{itemize}

\end{Rem}

\subsubsection{Applications to functional convergence}

For $\bb{t}=(t_1,t_2)\in [0,1]^2$, we set $\D_{\bb{t}}:=[0,t_1]\times [0,t_2]$ and write as before $\mathscr{L}_E[2](\D_{\bb{t}})$ for the projection of $\mathscr{L}_E(\D_{\bb{t}})$ on the second Wiener chaos. The next statement is a direct consequence of Theorem \ref{CLTd}. 
\begin{Thm}\label{Thm:C2} 
For every $\bb{t}\in [0,1]^2$, we set 
\begin{eqnarray*}
Y_E(\bb{t}) :=   4\pi E^{1/4} \, \mathscr{L}_E[2](\D_{\bb{t} }).
\end{eqnarray*}
For every integer $d\geq1$ and every collection of $\bb{t}_1,\ldots, \bb{t}_d \in [0,1]^2$, we have that, as $E\to \infty$,
\begin{eqnarray*}
\(Y_E(\bb{t}_1),\ldots, Y_E(\bb{t}_d)\) 
\Law  \mathcal{N}_d(0, \Sigma)
\end{eqnarray*}
where $\mathcal{N}_d(0, \Sigma)$ denotes the centred $d$-dimensional Gaussian distribution whose covariance matrix $\Sigma$ is given by \eqref{e:sigmalength} with $\CC_i = \partial\mathcal{D}_{\bb{t}_i}$. In particular, as $E\to\infty$, the random functions $\{Y_E(\bb{t}):\bb{t}\in [0,1]^2\}$ converge in the sense of finite-dimensional distributions to a total disorder random field. \end{Thm}
%

 
 \medskip
Combining Theorem \ref{Thm:C2} with a suitable tightness criterion  by Davydov and Zitikis 
\cite{DZ:08} for proving weak convergence of stochastic processes on $[0,1]^d$ (see Proposition \ref{PropTight2}), and with some moment estimates for suprema of stationary Gaussian fields (see Proposition \ref{MomChapter}), one deduces the next characterization of the high-energy behaviour of $X_E[2]$, as defined in \eqref{Eq:XE24R}. This settles Point \textbf{(II)} of {\bf Strategy S}, as outlined in Section \ref{ss:mainquestion}.
\begin{Cor}\label{Cor:SecondFCLT}
As $E \to \infty$, the field $\{X_E[2](\bb{t}):\bb{t}\in [0,1]^2\}$ weakly converges to zero in  
$\mathbf{D}_2$. 
\end{Cor}

We will now show how Corollary \ref{Cor:SecondFCLT}, together with the content of Theorem \ref{Thm:DomX4}, can be used in order to partially address Point {\bf (III)} of {\bf Strategy S}.

\subsection{Approximate tightness}\label{ss:approximate}

\subsubsection{Discretized nodal length process}\label{ss:discretized}
Let us first introduce some notation.  
For $K\geq1$, we indicate by $\Pi_{K}$ the partition of $[0,1]^2$ formed by the collection of squares of side length $2^{-K}$. 
For 
 every vector $i =(i_1,i_2)\in \{0,\ldots, 2^{K}\}^2$, we define the \textit{partition points} 
$\bb{p}_{i}(K,K):=(p_{i_1}(K),p_{i_2}(K))\in [0,1]^2$ by
\begin{eqnarray*}
p_{i_1}(K):=\frac{i_1}{2^K},     \quad 
p_{i_2}(K):=\frac{i_2}{2^{K}}, \quad i_1,i_2=0,1,\dots,2^K .
\end{eqnarray*}
For   $\bb{t}=(t_1,t_2)\in [0,1]^2$, we write  $i_{K,K}(\bb{t})=\(i_{1,K}(t_1),i_{2,K}(t_2)\)$ for the vector verifying 
\begin{eqnarray*}
p_{i_{1,K}(t_1)}\leq t_1 < p_{i_{1,K}(t_1)+1},  \quad 
p_{i_{2,K}(t_2)}\leq t_2 < p_{i_{2,K}(t_2)+1},
\end{eqnarray*}
that is, the vector $i_{K,K}(\bb{t})$ is such that $\bb{p}_{i_{K,K}(\bb{t})}(K,K)$ is the closest partition point to $\bb{t}$ on the left (by convention, if ${\bf t} = (t_1,1)$, with $t_1<1$, one sets $i_{K,K}(\bb{t})=\(i_{1,K}(t_1),2^K\)$, and analogous conventions are adopted when ${\bf t} = (1,t_2)$, $t_2<1$ and ${\bf t} = (1,1)$).

\medskip
We now introduce a notion of {\bf discretized nodal length}. 
\begin{Def}({Discretized nodal length field})
Let $K\geq 1$ be  an integer and   $\Pi_K$ a partition of $[0,1]^2$ as described above. For $\bb{t} \in [0,1]^2$, we define the {\bf discretized nodal length} at {\bf t} as
\begin{eqnarray*}
\mathscr{L}_E^{K}([0,t_1]\times [0,t_2]):= \mathscr{L}_E\([0, p_{i_{1,K}(t_1)}(K)] \times 
[0,p_{i_{2,K}(t_2)}(K)]\)
\end{eqnarray*}
and write $X_E^K$ for its normalized version:
\begin{eqnarray*}
X_E^K(\bb{t}) = \sqrt{\frac{512\pi}{\log E}} \( \mathscr{L}_E^{K}([0,t_1]\times [0,t_2]) - \E{\mathscr{L}_E^{K}([0,t_1]\times [0,t_2])} \).
\end{eqnarray*}
As usual, we write $X_E^K[q]$ for the projection of $X_E^K$ on the $q$th Wiener chaos and set $R_E^K= \sum_{q\geq 3} X_E^K[2q]$. 
\end{Def}
In short, the quantity $X_E^K(\bb{t})$ represents the normalized nodal length contained in the rectangle formed by the partition coordinates that are closest to $\bb{t}$, thus yielding a discrete approximation of $X_E(\bb{t})$. Moreover, $\mathscr{L}_E^K$ is $\Prob$-almost surely an element of $\mathbf{D}_2$.
The following result shows that 
there exists a suitable partition $\Pi_{K}$ of $[0,1]^2$ associated with a sequence $K=K(E)$ such that
the discretized residue process $R_E^K$ converges to zero uniformly on the unit square, thus showing a \textit{discretized} version of Point \textbf{(III)} of {\bf Strategy S}.
The proof of Theorem \ref{Thm:DiscretizedR} relies on a \textbf{planar chaining argument} inspired by \cite{DT:89, MW:11b}.
\begin{Thm}\label{Thm:DiscretizedR}
Let $\{K(E):E>0\}$ be a numerical sequence such that $K(E)\to \infty$ and $K(E) = o((\log E)^{1/10})$ as $E\to \infty$.
Then, for every $\eps>0$, 
\begin{eqnarray*}
\Prob\{ \sup_{\bb{t}\in [0,1]^2}
\left| R_E^{K(E)}(\bb{t})\right| >\eps
\} \to 0.
\end{eqnarray*}
\end{Thm}

Combining the findings of Theorem \ref{Thm:DiscretizedR},
Corollary \ref{Cor:SecondFCLT} and the weak convergence of $X_E[4]$ proved in \cite{PV:20} (see Theorem \ref{Thm:DomX4}), allows us to deduce the following weak convergence result for the   discretized nodal length process. 
\begin{Cor}\label{Cor:Discretized}
Let $\{K(E):E>0\}$ be a numerical sequence such that $K(E)\to \infty$ and $K(E) = o((\log E)^{1/10})$, as $E\to \infty$. 
Then, the normalized process $X_E^{K(E)}$ converges weakly to a standard Wiener sheet $\bb{W}$ on $[0,1]^2$ in the Skorohod space $\mathbf{D}_2$.
\end{Cor}
Corollary \ref{Cor:Discretized}
gives access to a number of new limit theorems dealing with specific functionals of the discretized nodal length process. Of particular interest is, for instance, the asymptotic behaviour of the
maximal discrepancy between the discretized nodal length and its expectation, given by the supremum of $X_E^K$. Such statistics 
provide global indications on how the nodal length process deviates from its mean and are intimately related to {\bf overcrowding estimates} and concentration inequalities. We refer the reader e.g. to \cite{Pr:20} for 
the study of such events in the framework of zero counts and nodal length associated with stationary Gaussian processes.

\medskip

\begin{Cor}\label{Cor:SupDiscretized}
Let $\{K(E):E>0\}$ be a numerical sequence such that $K(E)\to \infty$ and $K(E) = o((\log E)^{1/10})$ as $E\to \infty$. Then, as $E\to \infty$, we have that 
\begin{eqnarray*}
\sup_{\bb{t}\in [0,1]^2}
\left| X_E^{K(E)}(\bb{t})\right| \Law \sup_{\bb{t}\in [0,1]^2} |\bb{W}(\bb{t})|.
\end{eqnarray*}
\end{Cor}
To the best of our expertise, the probability distribution of the supremum of the Wiener sheet is not known. In \cite{PP:73}, the authors provide a number of explicit expressions for the probability distribution function of the supremum of  Wiener sheets restricted to the boundary of planar domains. For instance, the  following statement is a direct consequence of Corollary \ref{Cor:SupDiscretized} and \cite[Theorem 3]{PP:73}, yielding a closed formula for the asymptotic distribution function of the supremum of $X_E^{K}$ on the boundary of the unit square. We refer the reader to \cite{PP:73} for more examples in this direction. 
\begin{Cor}\label{Cor:Explicit}
Let $\{K(E):E>0\}$ be a numerical sequence such that $K(E)\to \infty$ and $K(E) = o((\log E)^{1/10})$ as $E\to \infty$. 
Then, for every $z\in \R$, we have that, as $E\to \infty$,
\begin{eqnarray*}
\Prob\{\sup_{\bb{t} \in \partial [0,1]^2}
\left| X_E^{K(E)}(\bb{t})\right| \leq z\}
\to 
 \Prob\{\sup_{\bb{t} \in \partial [0,1]^2}
|\bb{W}(\bb{t})| \leq z\}
= 1-3\Phi(-z)+ e^{4z^2}\Phi(-3z),
\end{eqnarray*}
where $\Phi(z):= \Prob\{N\leq z\}$, with $N$ being a standard Gaussian random variable.
\end{Cor}

\begin{Rem}
The findings described above are not sufficient to obtain a weak convergence result for the process $X_E$ and thus to fully address Part \textbf{(III)} of {\bf Strategy S}. 
Our main difficulty for directly dealing with the residual term $R_E$ (instead of its discretized version $R_E^K$) appears in the chaining argument used in the proof of Theorem \ref{Thm:DiscretizedR} and is essentially explained by the fact that the expectation of $X_E$ (which is of order $\sqrt{E/\log E}$) grows considerably faster than the normalizing factor $\log E$. Carrying out the planar chaining argument  with $R_E$ typically requires the quantity 
\begin{eqnarray*}
\left|\E{X_E(\bb{t})}
-\E{X_E(\bb{p}_{i_{K,K}(\bb{t})}(K,K))}\right|
\approx \frac{\sqrt{E}}{\sqrt{\log E}} \frac{1}{2^K}
\end{eqnarray*}
to be bounded, thus imposing $K=K(E)$ to be of logarithmic order. Such a requirement is however incompatible with the choice $o((\log E)^{1/10})$, as is needed in the above statements. Such a difficulty is eschewed  when dealing with the discretized versions, since in this case
$\E{X_E^K(\bb{t})}
=\E{X_E^K(\bb{p}_{i_{K,K}(\bb{t})}(K,K))}$
by construction of $X_E^K$, implying that the above difference is zero. 
One possible strategy for  providing a complete answer to \textbf{(III)} would be to prove that for every $\eps>0$
\begin{eqnarray*}
\Prob\{ \sup_{\bb{t}\in [0,1]^2}
\left| R_E^{K(E)}(\bb{t})-R_E(\bb{t})\right| >\eps
\} \to 0,
\end{eqnarray*} 
as $E\to \infty$, where $K(E)$ is as in Theorem \ref{Thm:DiscretizedR}.  However,  our arguments
 allow one to prove
 that such an asymptotic 
relation only holds pointwise in the $L^2(\Prob)$-sense
\begin{eqnarray*}
\E{\( R_E^{K(E)}(\bb{t})-R_E(\bb{t}) \)^2 } \leq c_1 \frac{1}{\log E}
\frac{1}{2^{K(E)}}  
\end{eqnarray*}
where $c_1>0$ is some absolute constant, thus converging to zero in view of our choice of $K(E)$ (see in particular Lemma \ref{Cheby}).
\end{Rem}

\subsubsection{Truncated nodal length process}
We also point out that our results on the second Wiener chaos are sufficient to formulate a  weak convergence result for   {\bf truncated nodal lengths} of increasing degree, defined as follows.
\begin{Def}({Truncated nodal length})
For an integer $N \geq 1$, we define the \textbf{truncated nodal length} of order $N$ by 
\begin{eqnarray*}
\mathscr{L}_E(\D;N) := \sum_{q= 0}^N\mathscr{L}_E[2q](\D).
\end{eqnarray*}
We write $X_E(\bb{t};N)$ for the normalized version of $\mathscr{L}_E([0,t_1]\times [0,t_2];N)$ and $R_E(\bb{t};N):= \sum_{q=3}^N
X_E[2q](\bb{t})$ for its chaotic projections of order $6$ to $N$. 
\end{Def}
The following result 
shows that the process $R_E(\bullet;N)$
converges to zero  for a well-chosen $N=N(E)$, as a consequence of the hypercontractivity property on Wiener chaoses. 
\begin{Prop}\label{Prop:TruncatedR}
 Let $N(E)=  \log_5( \log E)$. Then, as $E\to \infty$, the process 
$$
\{ R_E(\bb{t}; N(E)) : \bb{t}\in [0,1]^2\}
$$ 
converges weakly to zero in $\mathbf{D}_2$. 
\end{Prop}
Combining this result with the  weak convergence to zero of the second chaotic projections $X_E[2]$ (see Corollary  \ref{Cor:SecondFCLT}) and the weak convergence of $X_E[4]$, is sufficient to derive the following functional limit theorem for the {\bf truncated nodal length process} of order $N=N(E)$.
\begin{Cor}\label{Cor:TruncatedX}
Let $N(E)= \log_5(\log E)$. Then,
as $E\to \infty$, the process 
$$\{X_E(\bb{t};N(E)): \bb{t}\in [0,1]^2\}$$ 
converges weakly to a standard   Wiener sheet $\bb{W}$ on $[0,1]^2$ in $\mathbf{D}_2$.
\end{Cor}

%

\begin{Rem}
We point out that our findings outlined in Section \ref{s:main} naturally extend to the case of the {\bf nodal intersection point process} $\{\mathcal{N}_E([0,t_1]\times[0,t_2]):(t_1,t_2) \in [0,1]^2\}$  obtained by counting the nodal intersection points of two independent Berry random waves with the same frequency in the unit square -- see in particular \cite{NPR:19} for univariate results and \cite{PV:20} for multidimensional extensions on such a quantity.
\end{Rem}

The rest of the paper is devoted to the proof of our main results.

\section{Proof of the main results} 
%

\subsection{Proofs of Theorem \ref{AsyVar} and Theorem \ref{CLTd}} \label{Sec:Seg}
In order to prove Theorem \ref{AsyVar} and  Theorem \ref{CLTd}, we first prove their analog statements when the polygonal curves are replaced with straight line segments. More specifically, in Section \ref{LS}, we investigate the limiting covariance structure of $\phi_E$ when restricted to line segments,  by  
carefully taking into account all possible spatial configurations of two line segments. In Proposition \ref{PropCovSeg}, we show that,
in the high energy limit, this covariance is non-zero only when the line segments have a non-trivial intersection, that is, when the line segments are adjacent to each other.
In Proposition \ref{PropCLTLS}, we establish a multidimensional Gaussian limit theorem  for 
random vectors of the form $(\wt{\phi}_E(S_1),\ldots, \wt{\phi}_E(S_d))$, where $S_1,\ldots, S_d$ is a collection of line segments. 
 Our methods rely on both the Fourth Moment Theorem (see \cite[Theorem 5.2.7]{NP:12}) for proving normal approximations of chaotic sequences and its multidimensional counterpart (see \cite[Theorem 6.2.3]{NP:12}). 



\subsubsection{Study of line segments}\label{LS}
Let $S_1$ and $S_2$ be two line segments in $\R^2$, and consider the random variables $\phi_E(S_i)$, $i=1,2$. Our principal aim of this section is to   prove the following result.
\begin{Prop}\label{PropCovSeg} 
Let $S_1$ and $S_2$ be two line segments. Then, we have that, as $E\to \infty$,
\begin{eqnarray}\label{AsyCov}
\Cov{\phi_E(S_1)}{\phi_E(S_2)} = 
\frac{\lambda(S_1, S_2)}{16\pi^2\sqrt{E}}+o\left(\frac{1}{\sqrt{E}}\right) ,
\end{eqnarray}
where $\lambda(S_1, S_2)$ is the signed length  of $S_1\cap S_2$.
\end{Prop}
We start with some ancillary computations. Introducing normalized derivatives $
\tilde{\partial_i}  := \sqrt{2\pi^2 E} \partial_i  , \  i=1,2$ (where $\partial_i:=\partial_{x_i}=\partial/\partial x_i$)
and exploiting the definition of $\phi_E$ in  \eqref{Phi},  we have that, for every $\CC_1,\CC_2\in \mathscr{C}$ (writing $dx$ for $\mathcal{H}^1(dx)$ for brevity),
\begin{eqnarray}\label{CovA}
&&\Cov{\phi_E(\CC_1)}{\phi_E(\CC_2)} 
= \E{\phi_E(\CC_1) \phi_E(\CC_2)} \notag\\
&=& \frac{1}{128 \pi^2 E} \int_{\CC_1\times \CC_2} \E{B_E(x) \scal{\nabla B_E(x)}{\n_{\CC_1}(x)} B_E(y) \scal{\nabla B_E(y)}{\n_{\CC_2}(y)} } dxdy\notag\\
&=&\frac{1}{128 \pi^2 E} \sum_{i,j=1}^2 \int_{\CC_1\times \CC_2} 
\E{B_E(x)B_E(y)
\partial_iB_E(x)  \partial_jB_E(y) } \n_{\CC_1}^i(x)\n_{\CC_2}^j(y)
dxdy\notag\\
&=&\frac{2\pi^2E}{128 \pi^2 E} \sum_{i,j=1}^2 \int_{\CC_1\times \CC_2} 
\E{B_E(x)B_E(y)
\tilde{\partial_i}B_E(x)  \tilde{\partial_j}B_E(y) } \n_{\CC_1}^i(x)\n_{\CC_2}^j(y)dxdy\notag\\
&=: & \frac{1}{64} \sum_{i,j=1}^2 \int_{\CC_1\times \CC_2} \psi^E_{i,j}(x,y) \n_{\CC_1}^i(x)\n_{\CC_2}^j(y) dxdy,
\end{eqnarray}
where $\psi^E_{i,j}:\R^2\to\R$ is the function
\begin{eqnarray*}
\psi^E_{i,j}(x,y):=
\E{B_E(x)B_E(y)\tilde{\partial_i}B_E(x)  \tilde{\partial_j}B_E(y) }, \quad i,j=1,2.
\end{eqnarray*}
Since for every $x,y \in \R^2$ and every $i=1,2$, $(B_E(x),B_E(y),
\tilde{\partial_i}B_E(x),  \tilde{\partial_j}B_E(y))$ is a centered Gaussian vector, we 
apply Feynmann's formula (see for instance \cite[Proposition 4.15]{MP:11}) in order to simplify the above expression: for jointly Gaussian centered random variables $Z_1,\ldots,Z_4$, we have 
\begin{gather*}
\E{Z_1Z_2Z_3Z_4} = \gamma_{12}\gamma_{34}+\gamma_{13}\gamma_{24}+\gamma_{14}\gamma_{23} , \quad \gamma_{ij}:=\E{Z_iZ_j}.
\end{gather*}
Therefore, exploiting the covariance structure of the vector $(B_E(x),B_E(y),
\tilde{\partial_i}B_E(x),  \tilde{\partial_j}B_E(y))$ (see \cite[Lemma 3.1]{NPR:19}), we obtain
\begin{eqnarray}\label{PSI}
\psi^E_{i,j}(x,y)  
&=& \E{B_E(x)B_E(y)}\E{\tilde{\partial_i}B_E(x) \tilde{\partial_j}B_E(y)}+\E{B_E(x)\tilde{\partial_i}B_E(x) }\E{B_E(y)\tilde{\partial_j}B_E(y)} \notag\\ 
&& \quad+\E{B_E(x)\tilde{\partial_j}B_E(y)}\E{B_E(y)\tilde{\partial_i}B_E(x)}  \notag \\ 
&=&  r^E(x-y) \tilde{r}_{i,j}^E(x-y) - \tilde{r}^E_{0,j}(x-y)\tilde{r}^E_{0,i}(x-y),
\end{eqnarray}
where the second term is equal to zero by independence of 
$B_E(x)$ and $\nabla B_E(x)$ for every fixed $x\in \R^2$ and we set \begin{eqnarray*}
\tilde{r}_{i,j}^E(x-y) := \tilde{\partial}_{x_i}\tilde{\partial}_{y_i} r^E(x-y), \quad  i,j=0,1,2 
\end{eqnarray*}
where $r^E$ is as in \eqref{covE} and we adopt the convention that $\partial_0$ is the identity operator.
We now restrict \eqref{CovA} to the case where $\CC_i=S_i,i=1,2$ are straight line segments. Denoting by 
\begin{gather*}
B_E^{(\theta)}(x) := B_E(R_{\theta}x) , \quad   \quad 
\hat{B}_E^{(L)}(x):=B_E(x+L) , \quad x\in \R^2,  \theta\in [0,2\pi], L\in \R^2
\end{gather*}
where $R_{\theta}\in \mathcal{M}_{2\times 2}(\R)$ stands for the rotation matrix associated with angle $\theta$, it follows by isotropy and stationarity of   Berry's random field that
$B_E \eqLaw B_E^{(\theta)} $ and $ B_E \eqLaw \hat{B}_E^{(L)}
$, 
where $\eqLaw$ denotes equality in distribution of random fields.
These observations imply that for every choice of $x,y \in \R^2$, the pairs
$(B_E^{(\theta)}(x), B_E^{(\theta)}(y))$ and $(\hat{B}_E^{(L)}(x),\hat{B}_E^{(L)}(y))$ have the same distribution as $(B_E(x),B_E(y))$. As a consequence, we reduce our investigations to line segments given by the unit speed parametrizations
\begin{eqnarray}
&&\gamma_1: [0,\lambda_1] \to S_1, \quad t \mapsto \gamma_1(t) := te_1 \label{Seg1}\\
&&\gamma_2: [0,\lambda_2] \to S_2, \quad t \mapsto \gamma_2(t) := p+t\rho(\theta) \label{Seg2}
\end{eqnarray}
where $\lambda_i>0,i=1,2$ is the length of $S_i$, $e_i$ is the $i$-th canonical basis vector of $\R^2$, $p=(p_1,p_2)\in \R^2$ and $\rho(\theta):=(\cos\theta,\sin\theta)$ for $\theta\in [0,2\pi)$. 
\begin{Rem}
In view of the definition of $\phi_E$, it follows that whenever $S$ is a line segment given by a union of line segments $S = S' \cup S''$ sharing only one point, then  
$\phi_E(S) = \phi_E(S') + \phi_E(S'')$. It follows that, one can always express the covariance associated with arbitrary line segments as a linear combination involving only covariances associated with line segments that have the same origin. 
This implies that, in \eqref{Seg2}, we can consistently reduce to the case $p=(0,0)$, that is when $S_1$ and $S_2$ have the same origin, except when $S_1$ and $S_2$ are parallel but disjoint. Indeed, in order to see this, let us assume that $S_1=[P_1,Q_1]$ and $S_2=[P_2,Q_2]$ for points $P_1,P_2,Q_1,Q_2\in \R^2$ such that $p=P_2 \neq (0,0)$. (Here for $A,B\in \R^2$, we use the notation $[A,B]$ to indicate the line segment joining $A$ and $B$.) Denote by $\ell_1$ and $\ell_2$ the lines directed by $S_1$ and $S_2$ respectively and let $I= \ell_1 \cap \ell_2$. If $S_1\cap S_2 = \{I\}$, we consider the four line segments $[P_2,I], [I,Q_2], [P_1,I]$ and $[I,Q_1]$. By construction, we thus have $\phi_E(S_1)=\phi_E([P_1,I])+\phi_E([I,Q_1])$ and 
$\phi_E(S_2)=\phi_E([P_2,I])+\phi_E([I,Q_2])$, so that  
\begin{eqnarray*}
\Cov{\phi_E(S_1)}{\phi_E(S_2)} 
&=& \Cov{\phi_E([P_1,I])}{\phi_E([P_2,I])} 
+ \Cov{\phi_E([P_1,I])}{\phi_E([I,Q_2])}\\
&&+ \Cov{\phi_E([I,Q_1])}{\phi_E([P_2,I])}
+ \Cov{\phi_E([I,Q_1])}{\phi_E([I,Q_2])}
\end{eqnarray*}
and each of these covariances contains only line segments with common point $I$, which one can set to be the origin by translation invariance of the Berry random field. Similarly, if $S_1\cap S_2=\emptyset$, we consider the line segments $[I,P_1]$ and $[I,P_2]$. Then, again by linearity we can write (up to sign, which is determined by the orientation of $S_1$) 
$\phi_E([I,Q_1])-\phi_E([I,P_1])=\phi_E(S_1)$ and 
 $\phi_E([I,Q_2])-\phi_E([I,P_2])=\phi_E(S_2)$, so that
\begin{eqnarray*}
 \Cov{\phi_E(S_1)}{\phi_E(S_2)} &=& \Cov{\phi_E([I,Q_1])}{\phi_E([I,Q_2])}
 -\Cov{\phi_E([I,Q_1])}{\phi_E([I,P_2])}\\
&& -\Cov{\phi_E([I,P_1])}{\phi_E([I,Q_2])}+ \Cov{\phi_E([I,P_1])}{\phi_E([I,P_2])}, 
\end{eqnarray*}
and the covariances on the right hand side can be dealt with setting $I=(0,0)$ as before.
\begin{center}
\includegraphics[scale=0.3]{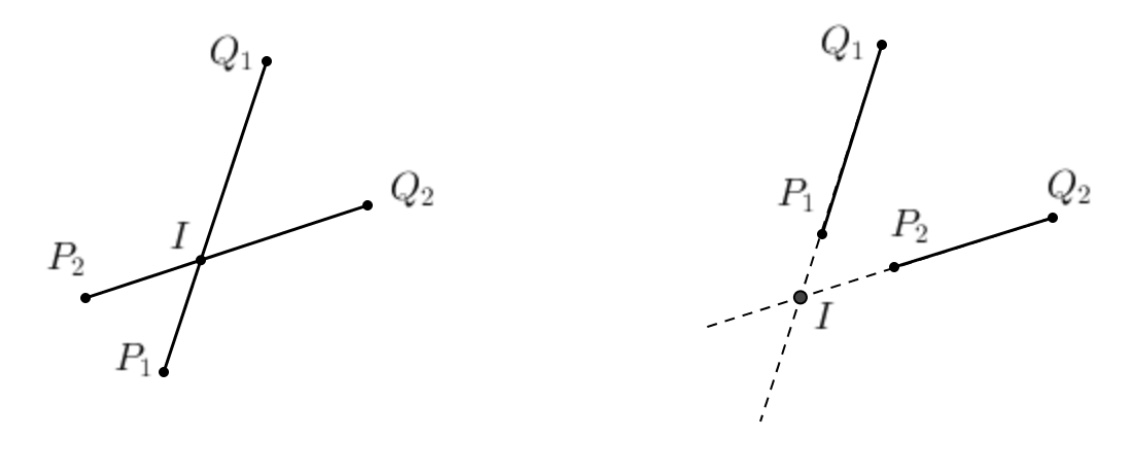}
\end{center}

In view of the above reductions, throughout this section, we will assume  that $S_1$ and $S_2$ are parametrized as in \eqref{Seg1} and \eqref{Seg2}, respectively with $p=(0,0)$ and $\theta \in [0,2\pi)$.
The fact that $\n_{S_1}(x) = e_2$ for every $x \in S_1$ and 
$\n_{S_2}(x) = \rho(\theta)^{\perp} = (-\sin\theta,\cos\theta)$ for every $x\in S_2$ yields
\begin{eqnarray}\label{General}
&&\Cov{\phi_E(S_1)}{\phi_E(S_2)} = \frac{1}{64}  \int_{S_1\times S_2} [\psi^E_{2,2}(x,y)\cos\theta- \psi^E_{2,1}(x,y)\sin\theta]dxdy \notag \\
&&= \frac{1}{64}  \int_{0}^{\lambda_1} dt \int_{0}^{\lambda_2} ds [\psi^E_{2,2}(\gamma_1(t),\gamma_2(s))\cos\theta- \psi^E_{2,1}(\gamma_1(t),\gamma_2(s))\sin\theta],
\end{eqnarray}
where $\psi_{i,j}^E$ is as in \eqref{PSI}.
From the parametrizations in \eqref{Seg1} and \eqref{Seg2}, it follows that 
\begin{eqnarray}\label{norm}
\norm{\gamma_1(t)-\gamma_2(s)}^2 
= \norm{te_1-s\rho(\theta)}^2 = t^2+s^2-2st\scal{e_1}{\rho(\theta)}
= t^2+s^2-2st\cos\theta.
\end{eqnarray}
\end{Rem}

\smallskip
Now, computations based on the explicit expressions of the functions $r^E, \tilde{r}^E_{i,j}$ for $i,j=0,1,2$ in terms of Bessel functions (see \cite[Lemma 3.1]{NPR:19}), lead to (for $\gamma_1(t)\neq \gamma_2(s)$)
\begin{eqnarray*}
\psi^E_{2,2}(\gamma_1(t),\gamma_2(s))& =&
J_0(\tau^E(t,s))\left( 
J_0(\tau^E(t,s))+J_2(\tau^E(t,s))\right) \\
&& - 2\frac{s^2 \sin^2\theta}{\norm{\gamma_1(t)-\gamma_2(s)}^2} \left(J_0(\tau^E(t,s))J_2(\tau^E(t,s))+J_1(\tau^E(t,s))^2\right)
\end{eqnarray*}
and 
\begin{eqnarray*}
\psi^E_{2,1}(\gamma_1(t),\gamma_2(s))& =&
2\frac{(t-s\cos\theta)s\sin\theta}{\norm{\gamma_1(t)-\gamma_2(s)}^2} \left(J_0(\tau^E(t,s))J_2(\tau^E(t,s))+J_1(\tau^E(t,s))^2\right),
\end{eqnarray*}
where we set $\tau^E(t,s):=2\pi\sqrt{E}\norm{\gamma_1(t)-\gamma_2(s)}$. 
As a consequence, by \eqref{norm}, we have that  
\begin{eqnarray*}
&&\psi^E_{2,2}(\gamma_1(t),\gamma_2(s))\cos\theta- \psi^E_{2,1}(\gamma_1(t),\gamma_2(s))\sin\theta \\
&=& \cos\theta J_0(\tau^E(t,s))
\left( J_0(\tau^E(t,s))+J_2(\tau^E(t,s))\right) \\
&&- \left( \frac{2s^2\sin^2\theta\cos\theta+2(t-s\cos\theta)s\sin^2\theta}{\norm{\gamma_1(t)-\gamma_2(s)}^2}\right) \left(J_0(\tau^E(t,s))J_2(\tau^E(t,s))+J_1(\tau^E(t,s))^2\right)\\
&=&\cos\theta J_0(\tau^E(t,s))
\left( J_0(\tau^E(t,s))+J_2(\tau^E(t,s))\right) \\
&&-   \frac{2ts\sin^2\theta}{t^2+s^2-2st\cos\theta}  \left(J_0(\tau^E(t,s))J_2(\tau^E(t,s))+J_1(\tau^E(t,s))^2\right) \\
&=&2\cos\theta \frac{J_0(\tau^E(t,s))
J_1(\tau^E(t,s))}{\tau^E(t,s))} 
\\ &&-   \frac{2ts\sin^2\theta}{t^2+s^2-2st\cos\theta} \left(J_0(\tau^E(t,s))J_2(\tau^E(t,s))+J_1(\tau^E(t,s))^2\right) ,
\end{eqnarray*}
where in the last line, we exploited the recurrence relation $J_{n+1}(x)+J_{n-1}(x)=2n J_{n}(x)/x, n> 0, x\in \R$ (see e.g. \cite[Equation (1.71.5)]{Sz:75}) implying the useful identity
\begin{eqnarray}\label{Rec}
J_0(x)+J_2(x) = 2\frac{J_1(x)}{x}.
\end{eqnarray}
Inserting this expression into \eqref{General}, we obtain that 
\begin{eqnarray*}
 &&\Cov{\phi_E(S_1)}{\phi_E(S_2)}  \\
&&= \frac{2\cos\theta}{64}  \int_{0}^{\lambda_1} dt \int_{0}^{\lambda_2} ds 
\frac{J_0(\tau^E(t,s))J_1(\tau^E(t,s))}{\tau^E(t,s)}
\\
&&-\frac{2\sin^2\theta}{64}\int_{0}^{\lambda_1} dt \int_{0}^{\lambda_2} ds 
\frac{ts\left(J_0(\tau^E(t,s))J_2(\tau^E(t,s))+J_1(\tau^E(t,s))^2\right)}{t^2+s^2-2st\cos\theta}   \\
&=:& A_E(\lambda_1,\lambda_2,\theta ) + B_E(\lambda_1,\lambda_2,\theta ) ,
\end{eqnarray*}
where 
\begin{gather}
 A_E(\lambda_1,\lambda_2,\theta ):= \frac{\cos\theta}{32}  \int_{0}^{\lambda_1} dt \int_{0}^{\lambda_2} ds 
\frac{J_0(\tau^E(t,s))J_1(\tau^E(t,s))}{\tau^E(t,s)}
\label{A}\\
 B_E(\lambda_1,\lambda_2,\theta ) :=-\frac{ \sin^2\theta}{32}\int_{0}^{\lambda_1} dt \int_{0}^{\lambda_2} ds 
\frac{ts\left(J_0(\tau^E(t,s))J_2(\tau^E(t,s))+J_1(\tau^E(t,s))^2\right)}{t^2+s^2-2st\cos\theta}  \label{B}
\end{gather}
and where we recall the notation $\tau^E(t,s)= 2\pi\sqrt{E} \sqrt{t^2+s^2-2st\cos\theta}$.
We note that if $\theta \in \{0,\pi\}$, then  $B_E(\lambda_1,\lambda_2,\theta )=0$, so that for parallel line segments we immediately deduce that 
\begin{eqnarray*}
\Cov{\phi_E(S_1)}{\phi_E(S_2)} = A_E(\lambda_1,\lambda_2,\theta ).
\end{eqnarray*}
We are now in position to prove Proposition 
\ref{PropCovSeg}.
\begin{proof}[Proof of Proposition \ref{PropCovSeg}]
Throughout the proof, we can and will assume without loss of generality, that $S_1$ and $S_2$ are both oriented in the same way. Indeed, if $S$ is a line segment with a given orientation, then $-S$ is the same line segment with opposite orientation to $S$, so that  $\phi_E(-S)=-\phi_E(S)$.

\medskip
In order to prove the statement, we distinguish two cases: \textit{(A)} $S_1$ and $S_2$ are parallel, and \textit{(B)} $S_1$ and $S_2$ are not parallel.

\medskip
\noindent\underline{\textit{Case (A):}}
We treat the case where $S_1$ and $S_2$ are parallel line segments. 
Let
$\gamma_1: t\in [a,b] \mapsto te_1$ and $\gamma_2: t\in [c,d]\mapsto te_1+Le_2$ where $0\leq a<b,0\leq c<d$ and $L\geq 0$ are fixed real numbers be the respective parametrizations of $S_1$ and $S_2$. Note that the case $L=0$ corresponds to the configuration where $S_1$ and $S_2$ are supported by the same line, whereas, the case $L>0$ corresponds to the case where 
$S_1$ and $S_2$ are supported by parallel distinct lines. 
We have that $\norm{\gamma_1(t)-\gamma_2(s)}^2  = \norm{(t,0)-(s,L)}^2=(t-s)^2+L^2$. Therefore, performing the linear change of variables $(u,v)= (t,t-s)$, we infer   
\begin{eqnarray}\label{COL}
&&\Cov{\phi_E(S_1)}{\phi_E(S_2)} \notag \\ 
&=&\frac{1}{32} \int_{a}^b dt\int_{c}^d ds \frac{J_0\left(2\pi\sqrt{E} 
\sqrt{(t-s)^2+L^2}\right)J_1\left(2\pi\sqrt{E} 
\sqrt{(t-s)^2+L^2}\right)}{2\pi\sqrt{E} 
\sqrt{(t-s)^2+L^2}}
\notag\\ 
&=& \frac{1}{32} \int_{a}^b du\int_{u-d}^{u-c} dv \frac{J_0\left(2\pi\sqrt{E} 
\sqrt{v^2+L^2}\right)J_1\left(2\pi\sqrt{E} 
\sqrt{v^2+L^2}\right)}{2\pi\sqrt{E} 
\sqrt{v^2+L^2}} \notag\\ 
&=& \frac{1}{32} \int_{a}^b du \frac{1}{2\pi\sqrt{E}}\int_{2\pi\sqrt{E}(u-d)}^{2\pi\sqrt{E}(u-c)} dv \frac{J_0\left( 
\sqrt{v^2+(2\pi\sqrt{E}L)^2}\right) J_1\left( \sqrt{v^2+(2\pi\sqrt{E}L)^2}\right) }{\sqrt{v^2+(2\pi\sqrt{E}L)^2}} \notag\\ 
&=:&\frac{1}{32} \int_{a}^b du K^E(u;L,c,d), 
\end{eqnarray}
where  we set 
\begin{gather}\label{KL}
K^E(u;L,c,d) :=\frac{1}{2\pi\sqrt{E}} \int_{2\pi\sqrt{E}(u-d)}^{2\pi\sqrt{E}(u-c)} dv 
\frac{J_0\left( 
\sqrt{v^2+(2\pi^2\sqrt{E}L)^2}\right) J_1\left( \sqrt{v^2+(2\pi\sqrt{E}L)^2}\right) }{\sqrt{v^2+(2\pi\sqrt{E}L)^2}}.
\end{gather} 
Note that, using \eqref{Rec} and the bound $|J_{\nu}(x)|\leq 1, \nu=0,1,2$ implies that 
\begin{eqnarray*}
\left|\frac{J_0(x)J_1(x)}{x}\right| = \frac{1}{2}\left|J_0(x)\left(J_0(x)+J_2(x)\right) \right| \leq 1
\end{eqnarray*}
so that 
\begin{eqnarray*}
\left|K^E(u;L,c,d)\right|  
\leq\frac{1}{2\pi\sqrt{E}}\int_{2\pi\sqrt{E}(u-d)}^{2\pi\sqrt{E}(u-c)} dv   = d-c,
\end{eqnarray*}
for every $u\in (a,b)$ and every $c<d$, so that $K^E(u;L,c,d)$ is integrable on $(a,b)$. 
We now study the two cases $L>0$ and $L=0$ separately. 

\medskip
\noindent\underline{\textit{Case (A.1): $L>0$}.} Fix $L>0$. 
We show that $K^E(u;L,c,d) = o(1/\sqrt{E})$ as $E \to \infty$ uniformly for  $u \in (a,b)$. Indeed, using the fact that $|J_{\nu}(x)| = O(x^{-1/2})$ for $x>0$ and $\nu=0,1,2$, we infer   from \eqref{KL} that 
for every $u \in (a,b)$, 
\begin{eqnarray*}
&&\sqrt{E}\left|K^E(u;L,c,d)\right|  \\
&&\leq
\frac{O(1)}{2\pi} \int_{2\pi\sqrt{E}(u-d)}^{2\pi\sqrt{E}(u-c)} dv 
\frac{1 }{v^2+(2\pi\sqrt{E}L)^2}
\leq \frac{O(1)}{2\pi} \int_{2\pi\sqrt{E}(u-d)}^{2\pi\sqrt{E}(u-c)} dv 
\frac{1 }{E} = O(E^{-1/2}),
\end{eqnarray*}
where the constant involved in the 'big-O' notation does not depend on $u$. Thus, $$\sqrt{E}K^E(u;L,c,d)\to 0$$ as $E \to \infty$ uniformly on $(a,b)$, and therefore we infer from \eqref{COL}
\begin{eqnarray*}
\Cov{\phi_E(S_1)}{\phi_E(S_2)} = O\left(\frac{1}{E}\right),
\end{eqnarray*}
as $E\to\infty$, which gives the desired conclusion.

\medskip
\noindent\underline{\textit{Case (A.2): $L=0$}.}
Setting $L=0$ in \eqref{KL}, we obtain 
\begin{eqnarray*}
K^E(u;0,c,d) = 
\frac{1}{2\pi\sqrt{E}} \int_{2\pi\sqrt{E}(u-d)}^{2\pi\sqrt{E}(u-c)} dv 
\frac{J_0(|v|) J_1(|v|)}{|v|}.
\end{eqnarray*} 
In order to show \eqref{AsyCov}, we treat the   two cases (i) $ [a,b]\cap[c,d] = \emptyset$  and 
(ii) $ [a,b]\cap[c,d] \neq \emptyset$. 
We start by case (i). This is the case when $a<b<c<d$ or $c<d<a<b$. We only treat the case $a<b<c<d$ as the other case is dealt with similarly. The assumption $a<b<c<d$ implies that 
$u-c<0$ and $u-d<0$ for every $u \in (a,b)$. Then, using the fact that $|J_{\nu}(x)| =O(x^{-1/2})$ for $x>0$, we have 
\begin{eqnarray*}
\sqrt{E} \left|K^E(u;0,c,d) \right| 
&\leq& \frac{1}{2\pi } \int_{2\pi\sqrt{E}(u-d)}^{2\pi\sqrt{E}(u-c)} dv 
\frac{|J_0(|v|) J_1(|v|)|}{|v|} \\
&=& \frac{O(1)}{2\pi } \int_{2\pi\sqrt{E}(u-d)}^{2\pi\sqrt{E}(u-c)}  
\frac{dv}{v^2}
= O\left(\frac{1}{\sqrt{E}}\right) \left(\frac{1}{c-u}-\frac{1}{d-u}\right),
\end{eqnarray*}
which goes to zero as  $E \to \infty$, uniformly for $u\in (a,b)$. Thus, from \eqref{COL} it follows that in this case 
\begin{eqnarray*}
\Cov{\phi_E(S_1)}{\phi_E(S_2)}
= O\left(\frac{1}{E}\right),
\end{eqnarray*}
which implies \eqref{AsyCov}.

\noindent
We now study the case (ii) and start with the special case $(c,d)=(a,b)$, that is when $S_1=S_2$. 
From \eqref{COL}, we write 
\begin{eqnarray*}
\sqrt{E} \Cov{\phi(S_1)}{\phi(S_1)}
 =
\frac{1}{32} \int_{a}^b du   \sqrt{E} K^E(u;0,a,b)
\end{eqnarray*}
with 
\begin{gather*}
\sqrt{E}K^E(u;0,a,b) = 
\frac{1}{2\pi} \int_{2\pi\sqrt{E}(u-b)}^{2\pi\sqrt{E}(u-a)} dv 
\frac{J_0(|v|) J_1(|v|)}{|v|}
= \frac{1}{2\pi} \int_{2\pi\sqrt{E}(u-b)}^{2\pi\sqrt{E}(u-a)} dv 
\frac{J_0(v) J_1(v)}{v},
\end{gather*}
where we used that $J_0$ is even and $J_1$ is odd. Now computations based on differentiation of Bessel functions imply that $\frac{d}{dv}[v(J_0(v)^2+J_1(v)^2)-J_0(v)J_1(v)] = J_0(v)J_1(v)/v$, so that
\begin{eqnarray*}
\sqrt{E}K^E(u;0,a,b) 
= \frac{1}{2\pi}\left[v(J_0(v)^2+J_1(v)^2)-J_0(v)J_1(v)\right]_{2\pi\sqrt{E}(u-b)}^{2\pi\sqrt{E}(u-a)}
\end{eqnarray*}
and therefore
\begin{eqnarray}\label{int2}
\sqrt{E} \Cov{\phi_E(S_1)}{\phi_E(S_1)}
&=& \frac{1}{64\pi}\int_a^b du \left[v(J_0(v)^2+J_1(v)^2) \right]_{2\pi\sqrt{E}(u-b)}^{2\pi\sqrt{E}(u-a)}\notag\\
&&- \frac{1}{64\pi}\int_a^b du \left[J_0(v)J_1(v)\right]_{2\pi\sqrt{E}(u-b)}^{2\pi\sqrt{E}(u-a)}.
\end{eqnarray}
For the first term, we use the dominated convergence theorem: since $|J_{\nu}(x)|\leq C_{\nu}x^{-1/2}, x>0$ for some constant $C_{\nu}>0$, it follows that 
\begin{eqnarray*}
\left| \left[v(J_0(v)^2+J_1(v)^2) \right]_{2\pi\sqrt{E}(u-b)}^{2\pi\sqrt{E}(u-a)}\right|
\leq 2[C_0^2+C_1^2]
\end{eqnarray*}
which is integrable on $(a,b)$. Setting $\mathfrak{f}(v):=v(J_0(v)^2+J_1(v)^2)$, we have $\mathfrak{f}(-v)=-\mathfrak{f}(v)$ and 
\begin{eqnarray*}
\left| \left[v(J_0(v)^2+J_1(v)^2) \right]_{2\pi\sqrt{E}(u-b)}^{2\pi\sqrt{E}(u-a)}\right|
= \mathfrak{f}(2\pi\sqrt{E}(u-a))-\mathfrak{f}(2\pi\sqrt{E}(u-b)),
\end{eqnarray*}
so that 
\begin{eqnarray*}
\lim_{E\to \infty} \left| \left[v(J_0(v)^2+J_1(v)^2) \right]_{2\pi\sqrt{E}(u-b)}^{2\pi\sqrt{E}(u-a)}\right|
=  \lim_{E\to \infty} 2\mathfrak{f}(2\pi\sqrt{E}(u-a)) = \lim_{y\to \infty}2\mathfrak{f}(y)
\end{eqnarray*}
since $u-a>0$.
Now, the asymptotic expansion of Bessel functions (see for instance \cite{Kr:14})
\begin{eqnarray}\label{Eq:BesselAsy}
J_{\nu}(y) = \sqrt{\frac{2}{\pi y}}\cos(y-\omega_{\nu}) + O(y^{-3/2}), \quad  \omega_{\nu}:= (2\nu+1)\frac{\pi}{4}, \quad y\to\infty
\end{eqnarray}
yield
\begin{eqnarray*}
2\mathfrak{f}(y) 
\sim 2y \frac{2}{\pi y} \left[\cos\left(x-\frac{\pi}{4}\right)^2+\cos\left(x-\frac{3\pi}{4}\right)^2\right] = 2y \frac{2}{\pi y} \left[\cos\left(x-\frac{\pi}{4}\right)^2+\sin\left(x-\frac{\pi}{4}\right)^2\right]
=\frac{4}{\pi}
\end{eqnarray*}
as $y\to \infty$.
 Thus, by dominated convergence, we obtain 
\begin{eqnarray*}
\frac{1}{64\pi}\int_a^b du \left[v(J_0(v)^2+J_1(v)^2) \right]_{2\pi\sqrt{E}(u-b)}^{2\pi\sqrt{E}(u-a)} \to \frac{1}{64\pi}\int_a^b \frac{4}{\pi}du = \frac{1}{16\pi^2}(b-a)
\end{eqnarray*}
as $E\to\infty$.
For the remainder term in \eqref{int2}, we use the bound $|J_0(x)|\leq C_0x^{-1/2}, x\geq 0$ and $|J_1(x)| \leq 1$ to obtain  
\begin{eqnarray*}
\int_a^b du \left[J_0(v)J_1(v)\right]_{2\pi\sqrt{E}(u-b)}^{2\pi\sqrt{E}(u-a)}
&\leq& \int_{a}^b \left|J_0(2\pi\sqrt{E}(u-a))\right|
\left|J_1(2\pi\sqrt{E}(u-a))\right|\\
&&\quad+ \left|J_0(2\pi\sqrt{E}(b-u))\right|
\left|J_1(2\pi\sqrt{E}(b-u))\right| du \\
&\leq&c\int_{a}^b \frac{1}{E^{1/4}\sqrt{u-a}}+\frac{1}{E^{1/4}\sqrt{b-u}} \leq c\frac{1}{E^{1/4}}\sqrt{b-a}
\end{eqnarray*}
for some constant $c>0$. This proves that 
\begin{eqnarray*}\Cov{\phi_E(S_1)}{\phi_E(S_1)} = \frac{1}{16\pi^2}\frac{b-a}{\sqrt{E}} + o\left(\frac{1}{\sqrt{E}}\right).
\end{eqnarray*}

\medskip
\noindent
Let us now assume that $S_1\neq S_2$ but $S_1\cap S_2\neq \emptyset$.
Without loss of generality, assume that $0<a<c\leq b<d$, that is $S_1\cap S_2=[c,b]\times \{0\}$. 
Exploiting the linearity of $\phi_E$, we write
\begin{eqnarray*}
\phi_E(S_2)= \phi_E([c,b]\times \{0\})+
\phi_E([b,d]\times \{0\})
\end{eqnarray*}
and use   the previous observations to obtain
\begin{eqnarray*}
&&\Cov{\phi_E(S_1)}{\phi_E(S_2)}\\
&=& \Cov{\phi_E([a,c]\times \{0\})}{\phi_E([c,b]\times\{0\})}
+  \Cov{\phi_E([a,c]\times \{0\})}{\phi_E([b,d]\times\{0\})}\\
&&\quad+  \Cov{\phi_E([c,b]\times \{0\})}{\phi_E([c,b]\times\{0\})} 
+  \Cov{\phi_E([c,b]\times \{0\})}{\phi_E([b,d]\times\{0\})}\\
&=& \Cov{\phi_E([c,b]\times \{0\})}{\phi_E([c,b]\times\{0\})} + o\left(\frac{1}{\sqrt{E}}\right)\\
&=&  
\frac{ \lambda(S_1, S_2)}{16\pi^2\sqrt{E}}  + o\left(\frac{1}{\sqrt{E}}\right),
\end{eqnarray*}
which is \eqref{AsyCov}.

\medskip
\noindent\underline{\textit{Case (B):}}
We now treat the case where the line segments are not parallel. We will use the 
parametrizations \eqref{Seg1} and \eqref{Seg2} of $S_1$ and $S_2$ respectively with 
\begin{eqnarray*}
p=(0,0), \quad 
\theta \in [0,2\pi) \setminus \{0,\pi\}.
\end{eqnarray*}
 Moreover, in this case  
\begin{eqnarray}\label{AB}
\Cov{\phi_E(S_1)}{\phi_E(S_2)}
= A_E(\lambda_1,\lambda_2,\theta)+ B_E(\lambda_1,\lambda_2,\theta)
\end{eqnarray}
where $ A_E(\lambda_1,\lambda_2,\theta)$ and $ B_E(\lambda_1,\lambda_2,\theta)$ are given in \eqref{A} and \eqref{B}, respectively. We show that both the contributions of $A_E(\lambda_1,\lambda_2,\theta)$ and $B_E(\lambda_1,\lambda_2,\theta)$ to the covariance are of order $o(E^{-1/2})$ in the high-energy regime.
By \eqref{Rec}, we can write
\begin{eqnarray*}
A_E(\lambda_1,\lambda_2,\theta )= \frac{\cos\theta}{64}  \int_{0}^{\lambda_1} dt \int_{0}^{\lambda_2} ds 
J_0(\tau^E(t,s))
\left( J_0(\tau^E(t,s))+J_2(\tau^E(t,s))\right)
\end{eqnarray*}
where we recall $\tau^E(t,s)= 2\pi\sqrt{E} \sqrt{t^2+s^2-2st\cos\theta}$. Passing to polar coordinates  $(t,s)=(\rho \cos\phi,\rho\sin\phi)$, we have 
\begin{eqnarray}\label{Tau}
\tau^E(\rho\cos\phi,\rho\sin\phi)&=&2\pi\sqrt{E} \sqrt{\rho^2-2\rho^2\sin\phi\cos\phi\cos\theta}
=2\pi\sqrt{E} \rho\sqrt{1-\sin(2\phi)\cos\theta}\notag\\
&=:& \tilde{\tau}^E(\rho,\phi).
\end{eqnarray}
We note that $\tilde{\tau}^E(\rho,\phi)>0$  and $\cos\theta\neq 0$ for $\theta \in [0,2\pi)\setminus \{0,\pi\}$. 
Using polar coordinates $(\rho,\phi)$ on rectangle $[0,\lambda_1]\times [0,\lambda_2]$ and the fact that the line joining the origin and the point $(\lambda_1,\lambda_2)$ forms an angle of $\arctan(\lambda_2/\lambda_1)$ shows that the range of integration 
 is parametrized  according to
\begin{eqnarray}\label{split}
\int_{0}^{\lambda_1}dt\int_{0}^{\lambda_2}ds = \int_{0}^{\alpha_{1,2}} d\phi \int_{0}^{\lambda_1/\cos\phi}\rho d\rho
+ \int_{\alpha_{1,2}}^{\pi/2} d\phi \int_{0}^{\lambda_2/\sin\phi}\rho d\rho ,
\end{eqnarray}
where we set $\alpha_{1,2}:=\arctan(\lambda_2/\lambda_1) \in (0, \pi/2)$.
We split the integral 
\begin{eqnarray*}
&&A_E(\lambda_1,\lambda_2,\theta )
= \frac{\cos\theta}{64}  \int_{0}^{\lambda_1} dt \int_{0}^{\lambda_2} ds 
J_0(\tau^E(t,s))
\left( J_0(\tau^E(t,s))+J_2(\tau^E(t,s))\right) \\
&=& \frac{\cos\theta}{64} \int_{0}^{\alpha_{1,2}} d\phi \int_{0}^{\lambda_1/\cos\phi}\rho d\rho 
J_0(\tilde{\tau}^E(\rho,\phi))
\left( J_0(\tilde{\tau}^E(\rho,\phi))+J_2(\tilde{\tau}^E(\rho,\phi))\right)\\
&&+ \frac{\cos\theta}{64} \int_{\alpha_{1,2}}^{\pi/2} d\phi \int_{0}^{\lambda_2/\sin\phi}\rho d\rho
J_0(\tilde{\tau}^E(\rho,\phi))
\left( J_0(\tilde{\tau}^E(\rho,\phi))+J_2(\tilde{\tau}^E(\rho,\phi))\right)\\
&=:&  A_{E,1}(\lambda_1,\lambda_2,\theta ) + A_{E,2}(\lambda_1,\lambda_2,\theta ). 
\end{eqnarray*}
We focus on the term $A_{E,1}(\lambda_1,\lambda_2,\theta)$. 
For fixed $\phi\in (0,\alpha_{1,2})$, we perform the change of variable $ \psi=\tilde{\tau}^E(\rho,\phi)$ with $d\psi =   \tilde{\tau}^E(1,\phi)d\rho $, yielding 
\begin{eqnarray}\label{A1}
A_{E,1}(\lambda_1,\lambda_2,\theta )
&=&\frac{\cos\theta}{64} \int_{0}^{\alpha_{1,2}} \frac{d\phi}{(\tilde{\tau}^E(1,\phi))^2} \int_{0}^{\frac{\tilde{\tau}^E(1,\phi)\lambda_1}{\cos\phi}}  d\psi \psi
J_0(\psi)
\left( J_0(\psi)+J_2(\psi)\right)\notag\\
&=:&\frac{\cos\theta}{64} \int_{0}^{\alpha_{1,2}}  d\phi 
K^E(\phi;\lambda_1,\theta),
\end{eqnarray}
with 
\begin{gather}\label{KE}
 K^E(\phi;\lambda_1,\theta)
= \frac{1}{(\tilde{\tau}^E(1,\phi))^2} \int_{0}^{\frac{\tilde{\tau}^E(1,\phi)\lambda_1}{\cos\phi}}  d\psi \psi
J_0(\psi)
\left( J_0(\psi)+J_2(\psi)\right) \notag\\
= \frac{2}{(\tilde{\tau}^E(1,\phi))^2} \int_{0}^{\frac{\tilde{\tau}^E(1,\phi)\lambda_1}{\cos\phi}}  d\psi 
J_0(\psi)J_1(\psi) 
=\frac{2}{(\tilde{\tau}^E(1,\phi))^2} 
\left[ -\frac{J_0(\psi)^2}{2}\right]_0^{\frac{\tilde{\tau}^E(1,\phi)\lambda_1}{\cos\phi}}  ,
\end{gather}
where we used \eqref{Rec} and the fact that $\frac{d}{d\psi} J_0(\psi) =-J_1(\psi)$. 
Thus, it follows that (since $|J_0(x)|\leq 1$)
\begin{eqnarray*}
\left|K^E(\phi;\lambda_1,\theta) \right| \leq 
\frac{1}{(\tilde{\tau}^E(1,\phi))^2}\left| J_0\left(\frac{\tilde{\tau}^E(1,\phi)\lambda_1}{\cos\phi}\right)^2 - J_0(0)^2\right|
\leq \frac{2}{(\tilde{\tau}^E(1,\phi))^2},
\end{eqnarray*}
so that by \eqref{A1}
\begin{gather*}
\left| A_{E,1}(\lambda_1,\lambda_2,\theta ) \right| 
\leq \int_0^{\alpha_{1,2}} d\phi \frac{2}{(\tilde{\tau}^E(1,\phi))^2}
= \frac{2}{4\pi^2 E}\int_0^{\alpha_{1,2}}  \frac{d\phi}{1-\sin(2\phi)\cos(\theta)} = O(E^{-1})
\end{gather*} 
where the last upper bound is justified by the reverse triangular inequality $|x-y|\geq ||x|-|y||$, the assumption $|\cos\theta| \neq1$ and 
\begin{eqnarray}\label{Eq:reverse}
\int_0^{\alpha_{1,2}}
\frac{d\phi}{|1-\sin(2\phi)\cos\theta|}
\leq
\int_0^{\alpha_{1,2}}
\frac{d\phi}{1-|\sin(2\phi)||\cos\theta|}
\leq \frac{\alpha_{1,2}}{1-|\cos\theta|}<\infty.
\end{eqnarray}
Arguing similarly for the term $A_{E,2}(\lambda_1,\lambda_2,\theta)$, we obtain that $|A_{E,2}(\lambda_1,\lambda_2,\theta)|=O(E^{-1})$, so that $|A_{E}(\lambda_1,\lambda_2,\theta)|=O(E^{-1})$ as $E\to \infty$. We now treat the term $B_{E}(\lambda_1,\lambda_2,\theta)$.
From \eqref{B}, we have 
\begin{eqnarray*}
B_E(\lambda_1,\lambda_2,\theta )  := -\frac{\sin^2\theta}{32}\int_{0}^{\lambda_1} dt \int_{0}^{\lambda_2} ds 
\frac{ts\left(J_0(\tau^E(t,s))J_2(\tau^E(t,s))+J_1(\tau^E(t,s))^2\right)}{t^2+s^2-2st\cos\theta}  
\end{eqnarray*}
Passing to polar coordinates and using \eqref{split}, we write 
\begin{eqnarray*}
B_E(\lambda_1,\lambda_2,\theta ) =
B_{E,1}(\lambda_1,\lambda_2,\theta )
+B_{E,2}(\lambda_1,\lambda_2,\theta ),
\end{eqnarray*} 
where
\begin{eqnarray*}
&&B_{E,1}(\lambda_1,\lambda_2,\theta ) \\&&:= -\frac{\sin^2\theta}{32}\int_{0}^{\alpha_{1,2}} \frac{d\phi}{(\tilde{\tau}^E(1,\phi))^2} 
\frac{\sin\phi\cos\phi}{1-\sin(2\phi)\cos\theta}
\int_{0}^{\frac{\tilde{\tau}^E(1,\phi)\lambda_1}{\cos\phi}}  d\psi \psi
\left(J_0(\psi)J_2(\psi)+J_1(\psi)^2\right), 
\end{eqnarray*}
and 
\begin{eqnarray*}
&& B_{E,2}(\lambda_1,\lambda_2,\theta ) \\
&& := -\frac{\sin^2\theta}{32}\int_{\alpha_{1,2}}^{\pi/2} \frac{d\phi}{(\tilde{\tau}^E(1,\phi))^2} 
\frac{\sin\phi\cos\phi}{1-\sin(2\phi)\cos\theta}
\int_{0}^{\frac{\tilde{\tau}^E(1,\phi)\lambda_2}{\sin\phi}}  d\psi \psi
\left(J_0(\psi)J_2(\psi)+J_1(\psi)^2\right). 
\end{eqnarray*}
We treat the first term $B_{E,1}(\lambda_1,\lambda_2,\theta )$. 
We write 
\begin{eqnarray}\label{intM}
B_{E,1}(\lambda_1,\lambda_2,\theta ) 
:= -\frac{\sin^2\theta}{32}\int_{0}^{\alpha_{1,2}} 
d\phi M^E(\phi;\lambda_1,\theta),
\end{eqnarray}
where
\begin{eqnarray*}
M^E(\phi;\lambda_1,\theta)
:= \frac{1}{(\tilde{\tau}^E(1,\phi))^2} 
\frac{\sin\phi\cos\phi}{1-\sin(2\phi)\cos\theta}
\int_{0}^{\frac{\tilde{\tau}^E(1,\phi)\lambda_1}{\cos\phi}}  d\psi \psi
\left(J_0(\psi)J_2(\psi)+J_1(\psi)^2\right).
\end{eqnarray*}
Using the asymptotic expansion of Bessel functions in \eqref{Eq:BesselAsy} yields
\begin{eqnarray*}
\psi \left(J_0(\psi)J_2(\psi)+J_1(\psi)^2\right)
= \frac{2}{\pi} \cos\(2\psi+\frac{\pi}{2}\) + O(\psi^{-1})
=\frac{2}{\pi} \sin(2\psi)+O(\psi^{-1})
\end{eqnarray*}
as $\psi\to \infty$, so that for large $E$
\begin{eqnarray*}
&&\int_{0}^{\frac{\tilde{\tau}^E(1,\phi)\lambda_2}{\cos\phi}}  d\psi \psi 
\left(J_0(\psi)J_2(\psi)+J_1(\psi)^2\right)\\
&=& \int_0^1 
d\psi \psi \left(J_0(\psi)J_2(\psi)+J_1(\psi)^2\right)
+ \int_1^{\frac{\tilde{\tau}^E(1,\phi)\lambda_2}{\cos\phi}} d\psi \psi \left(J_0(\psi)J_2(\psi)+J_1(\psi)^2\right) \\
&=& O(1)+
\frac{2}{\pi}\int_{1}^{\frac{\tilde{\tau}^E(1,\phi)\lambda_2}{\cos\phi}} d\psi  \left( \sin(2\psi) +O(\psi^{-1})\right)\\
&=&  \frac{1}{\pi}  \cos\left(\frac{2\tilde{\tau}^E(1,\phi)\lambda_2}{\cos\phi}\right) + O(1)\left(1+\log\left(\frac{\tilde{\tau}^E(1,\phi)\lambda_2}{\cos\phi}\right)\right) \\
&=& O(1)\left(1+\log\left(\frac{\tilde{\tau}^E(1,\phi)\lambda_2}{\cos\phi}\right)\right).
\end{eqnarray*}
Therefore, we conclude by \eqref{intM}  
\begin{eqnarray*}
B_{E,1}(\lambda_1,\lambda_2,\theta)
&=& -\frac{\sin^2\theta}{32}\int_{0}^{\alpha_{1,2}} \frac{d\phi}{(\tilde{\tau}^E(1,\phi))^2} 
\frac{\sin\phi\cos\phi}{1-\sin(2\phi)\cos\theta}
\{O(1)\left(1+\log\left(\frac{\tilde{\tau}^E(1,\phi)\lambda_2}{\cos\phi}\right)\right)\} \\
&=& -\frac{\sin^2\theta}{32\pi} \frac{O(1)}{4\pi^2E} 
\int_0^{\alpha_{1,2}} 
d\phi \frac{\sin\phi\cos\phi}{(1-\sin(2\phi)\cos\theta)^2} \\
&&  -\frac{\sin^2\theta}{32\pi} \frac{O(1)}{4\pi^2E} 
\int_0^{\alpha_{1,2}} 
d\phi \frac{\sin\phi\cos\phi}{(1-\sin(2\phi)\cos\theta)^2}\log\left( \tilde{\tau}^E(1,\phi)\lambda_2\right)\\
&&+\frac{\sin^2\theta}{32\pi} \frac{O(1)}{4\pi^2E} 
\int_0^{\alpha_{1,2}} 
d\phi \frac{\sin\phi\cos\phi}{(1-\sin(2\phi)\cos\theta)^2}\log\left( \cos\phi\right)\\
&=:&   b_E^1+b_E^2+b_E^3.
\end{eqnarray*}
Clearly we have $|b_E^1|=O(E^{-1})$ since (arguing similarly as in \eqref{Eq:reverse})
\begin{gather*}
\int_0^{\alpha_{1,2}} 
d\phi \left|\frac{\sin\phi\cos\phi}{(1-\sin(2\phi)\cos\theta)^2}\right| \leq \int_0^{\alpha_{1,2}} 
 \frac{d\phi}{(1-|\cos\theta|)^2} 	= \frac{\alpha_{1,2}}{(1-|\cos\theta|)^2} <\infty
\end{gather*}
since  $|\cos\theta|\neq 1$. Let us now consider the term $b_E^2$. Using \eqref{Tau}, we write  
\begin{eqnarray*}
\log\left(\tilde{\tau}^E(1,\phi)\lambda_2\right)
&=& \log\left(2\pi\lambda_2\sqrt{E}\sqrt{1-\sin(2\phi)\cos\theta}\right) \\
&=&2^{-1}\log E+\log\left(2\pi\lambda_2\sqrt{1-\sin(2\phi)\cos\theta}\right) = O(\log E)+O(1),
\end{eqnarray*}
where we used the fact that the map $\phi\mapsto \log\left(2\pi\lambda_2\sqrt{1-\sin(2\phi)\cos\theta}\right)$ is bounded.
 Thus, arguing as above shows that 
$|b_E^2| = O(E^{-1}+\log(E)/E)= O(\log(E)/E)$.
For the term $b_E^3$, we show that $|b_E^3|=O(E^{-1})$. Indeed, since $ \alpha_{1,2}=\arctan(\lambda_2/\lambda_1)\leq \pi/2$, we have 
\begin{eqnarray*}
\int_0^{\alpha_{1,2}} 
d\phi \left|\frac{\sin\phi\cos\phi}{(1-\sin(2\phi)\cos\theta)^2}\log\left( \cos\phi\right) \right|\leq
\frac{1}{(1-|\cos\theta|)^2}   
\int_0^{\pi/2} d\phi |\log(\cos \phi)| <\infty
\end{eqnarray*}
since it is straightforward to check that
\begin{eqnarray*}
\int_0^{\pi/2} d\phi |\log(\cos \phi)| = \frac{\pi \log2}{2}.
\end{eqnarray*}
Indeed, the last integral is obtained as follows: 
Since $\log(\cos\phi))\leq0$ on $(0,\pi/2)$, we have 
\begin{eqnarray*}
\int_0^{\pi/2} d\phi |\log(\cos \phi)|=
-\int_0^{\pi/2} d\phi \log(\cos \phi)=:-I
\end{eqnarray*}
By changing variable $u=\pi/2-\phi$, we have that 
$I= \int_0^{\pi/2} d\phi \log(\sin \phi) $ and also by symmetry $I= \int_{\pi/2}^{\pi} d\phi \log(\sin \phi)$. Therefore, we have 
\begin{eqnarray*}
2I &=& \int_{0}^{\pi/2} d\phi \log(\cos\phi\sin\phi)
= \int_0^{\pi/2} d\phi \left(\log(\sin2\phi)-\log2\right)\\
&=& \frac{1}{2}\int_0^{\pi}\log(\sin\phi)d\phi - \frac{\pi\log2}{2}
= \frac{1}{2}\times 2I -\frac{\pi\log2}{2} = I-\frac{\pi\log2}{2},
\end{eqnarray*}
so that $I= -\frac{\pi\log2}{2}$ as desired.  Combining the contributions of each of the terms $b_E^j,j=1,2,3$, we conclude that $B_{E,1}(\lambda_1,\lambda_2,\theta) = O(E^{-1}\log E)$. 
The analysis for $B_{E,2}(\lambda_1,\lambda_2,\theta)$ is done analogously, so that $B_{E,2}(\lambda_1,\lambda_2,\theta)=O(E^{-1}\log
E)$. We conclude from \eqref{AB} that, as $E\to\infty$,
\begin{gather*}
\Cov{\phi_E(S_1)}{\phi_E(S_2)}
= A_E(\lambda_1,\lambda_2,\theta)+ B_E(\lambda_1,\lambda_2,\theta) = O\left(\frac{\log E}{E}\right) = o\left(\frac{1}{\sqrt{E}}\right),
\end{gather*}
which proves the statement. This concludes the proof.
\end{proof}
We now prove the following multivariate Central Limit Theorem for the normalized random variables $\wt{\phi}_E(S)=4\pi E^{1/4}\phi_E(S)$. 
\begin{Prop}[Multidimensional CLT for line segments]\label{PropCLTLS}
For every integer $d\geq 1$ and every line segments $S_1,\ldots, S_d$, we have that, as $E\to \infty$, 
\begin{eqnarray*}
\( \wt{\phi}_{E}(S_1), \ldots, \wt{\phi}_{E}(S_d)\) 
\Law \mathcal{N}_d(0, \Sigma),
\end{eqnarray*}
where $\Sigma=\{\Sigma(i,j):i,j=1,\ldots,d\}$ is the $d\times d$ matrix defined by
\begin{eqnarray*}
\Sigma(i,j) :=  \lambda(S_i, S_j) \ , \quad i,j=1,\ldots, d,
\end{eqnarray*}
where $\lambda(S_i, S_j)$ is the signed length of $S_i\cap S_j$.
\end{Prop}
\begin{proof}
Using the fact that for every line segment $S$, $\wt{\phi}_E(S)$ is an element of the second Wiener chaos and we proved  that the covariances $\E{\wt{\phi}_E(S_1)\wt{\phi}_E(S_2)}$ converge to $\lambda(S_i,S_j)$ as $E\to \infty$ (see Proposition \ref{PropCovSeg}), it is sufficient to prove the statement for $d=1$, since in view of \cite[Theorem 6.2.3]{NP:12}, joint convergence is equivalent to marginal convergence for chaotic sequences. 
By invariance under rigid motions of the plane of  Berry's random wave model, we can assume without loss of generality, that $S_1=[0,L]\times \{0\}$ for  $L>0$. Using the fact that $\n_{S_1}=e_1$, we have
 \begin{gather}\label{I2}
\wt{\phi}_{E}(S_1)
= \frac{\pi E^{1/4}}{2} \int_0^L  
 B_E(x,0) \tilde{\partial}_2B_E(x,0) dx.
\end{gather}
We now represent $\wt{\phi}_{E}(S_1)$ as a multiple integral of order $2$ with respect to an isonormal Gaussian process on the Hilbert space $L^2([0,1], \lambda)$, where $\lambda$ denotes Lebesgue measure (see Remark \ref{r:chains} \textbf{(c)}).
For $(x,0) \in \R^2$, let $f_0^E(x,\cdot),f_2^E(x,\cdot): [0,1]\to \R$ be such that 
\begin{gather*}
B_E(x,0) = I_1(f_0^E(x,\cdot) ), \quad 
\tilde{\partial}_2B_E(x) = I_1(f_2^E(x,\cdot)),
\end{gather*}
where $I_{p}$ denotes the Wiener-It\^{o} isometry of order $p$.
Using the product formula for multiple integrals (see \cite[Theorem 2.7.10]{NP:12}) and independence, we can write
\begin{eqnarray*}
&&B_E(x,0) \tilde{\partial}_2 B_E(x,0)\\
&&= I_2\left(f_0^E(x,\cdot) \wt{\otimes} f_2^E(x,\cdot)\right)
+ I_0\left(f_0^E(x,\cdot) \wt{\otimes}_1 f_2^E(x,\cdot)\right)
= I_2\left(f_0^E(x,\cdot) \wt{\otimes} f_2^E(x,\cdot)\right),
\end{eqnarray*}
where the symbols $\otimes_r$ and $\wt{\otimes}_r$ denote the contraction operator of order $r$ and its symmetrization respectively (see \cite[Appendix B]{NP:12}. In particular, for $r=0$ and $r=1$, these are  given by (writing $\lambda(du)=du$)
\begin{eqnarray*}
f_0^E(x,\cdot) {\otimes}_0 f_2^E(x,\cdot) &=& f_0^E(x,\cdot) {\otimes} f_2^E(x,\cdot), \\
f_0^E(x,\cdot) {\otimes}_1 f_2^E(x,\cdot)
&=& \int_{0}^{1} f_0^E(x,u)f_2^E(x,u) du = \scal{f_0^E(x,\cdot)}{f_2^E(x,\cdot)}_{L^2([0,1], \lambda)} = 
0,
\end{eqnarray*}
where the last identity follows from the isometry property for Wiener integrals  and independence. It follows from \eqref{I2} that, 
\begin{gather*}
\wt{\phi}_{E}(S_1) 
= \frac{\pi E^{1/4}}{2} \int_0^L   I_2\left(f_0^E(x,\cdot) \wt{\otimes} f_2^E(x,\cdot)\right) dx  =: I_2\left(k^E\right), 
\end{gather*}
where ($\mathrm{Sym}$ denotes the symmetrization operator)
\begin{eqnarray}\label{kj}
k^E(u,v) &=& \frac{\pi E^{1/4}}{2} \mathrm{Sym}\{ \int_0^L   f_0^E(x,u)  f_2^E(x,v) dx\} \notag\\
&=& \frac{\pi E^{1/4}}{4} \{\int_0^L  f_0^E(x,u)  f_2^E(x,v) dx
+ \int_{0}^L   f_0^E(x,v)  f_2^E(x,u) dx \}.
\end{eqnarray}
 In order to show that $\wt{\phi}_{E}(S_1)$ satisfies a CLT as $E \to \infty$, it suffices to show that $\norm{k^E \otimes_1 k^E}_{L^2([0,1]^2, {\lambda}^{\otimes 2})}$ converges to zero as $E\to \infty$, in view of the Fourth Moment Theorem \cite[Theorem 5.2.7]{NP:12}.  
From \eqref{kj} it follows that 
\begin{eqnarray*}
&&  k^E \otimes_1  k^E(u,v) = \int_{0}^{1} dz k^E(u,z)k^E(v,z) \\
&=& \frac{\pi^2\sqrt{E}}{16} \int_{0}^{1} dz   \{\int_0^L  f_0^E(x,u)  f_2^E(x,z) dx
+ \int_0^L   f_0^E(x,z)  f_2^E(x,u) dx \}\\
&&\hspace{2cm}\times   \{\int_0^L  f_0^E(y,v)  f_2^E(y,z) dy
+ \int_0^L   f_0^E(y,z)  f_2^E(y,v) dy \},
\end{eqnarray*}
that is, after expanding, a sum of four terms, among which one of them   (ignoring multiplicative constants that are independent of $E$) is 
\begin{eqnarray*}
(u,v) &\mapsto& 
     \sqrt{E}\int_0^L dx \int_0^L dy   f_0^E(x,u)  f_0^E(y,v) \int_0^{1} dzf_2^E(x,z)     f_2^E(y,z)  \\
&=&   \sqrt{E} \int_0^L dx \int_0^L dy   f_0^E(x,u)  f_0^E(y,v) \E{\tilde{\partial}_2B_E(x,0) \tilde{\partial}_2 B_E(y,0)}\\
&=&\sqrt{E} \int_0^L dx \int_0^L dy   f_0^E(x,u)  f_0^E(y,v) \tilde{r}_{2,2}^E(x-y,0).
\end{eqnarray*}
From this, we compute the squared norm
\begin{eqnarray*}
\norm{ k ^E \otimes_1  k^E}_{L^2([0,1]^2, {\lambda}^{\otimes 2})}^2
= \int_{0}^{1} du \int_{0}^{1}dv \left[ k^E \otimes_1  k^E(u,v)\right]^2,
\end{eqnarray*}
which is given by a sum of $16$ terms that have all the same behaviour.  We will expose the details for one of them (which is representative of the difficulty), the others can be treated similarly. Exploiting once more the isometry property of Wiener integrals, one among them (corresponding to the computation above) is given by 
\begin{eqnarray}\label{integral}
E \int_{[0,L]^4}  dx_1\ldots dx_4 \  
\tilde{r}_{2,2}^E(x_1-x_2,0) 
   \tilde{r}_{2,2}^E(x_2-x_3,0) r^E(x_3-x_4,0) r^E(x_4-x_1,0)   .
\end{eqnarray}
We now show that the integral in \eqref{integral} converges to zero as $E\to \infty$. Performing the change of variables $(u_1,u_2,u_3,u_4) = (x_1-x_2,x_2-x_3,x_3-x_4,x_4)$ 
yields that the integral in \eqref{integral} is equal to 
\begin{eqnarray}
&& E \int_{0}^L  du_4  \int_{-u_4}^{L-u_4} du_3 
\int_{-(u_3+u_4)}^{L-(u_3+u_4)} du_2  
\int_{-(u_2+u_3+u_4)}^{L-(u_2+u_3+u_4)} du_1 \
\tilde{r}_{2,2}^E(u_1,0) \tilde{r}_{2,2}^E(u_2,0) r^E(u_3,0) \notag\\
&&\hspace{10cm} \cdot r^E(-u_1-u_2-u_3,0) \notag\\
&\leq& E L     \int_{-L}^{L} du_3 
\int_{-2L}^{2L} du_2  
\int_{-4L}^{4L}du_1 \
\left| \tilde{r}_{2,2}^E(u_3,0) \tilde{r}_{2,2}^E(u_2,0)   r^E(u_1,0)\right|, \label{Eq:intU}
\end{eqnarray}
where in the second line, we used the fact that  $|r^E(\cdot)| \leq 1$ and uniformly bounded the regions of integrations. 
Now using  \cite[Lemma 3.1]{NPR:19} and the  relation \eqref{Rec} yields
\begin{eqnarray*}
\tilde{r}_{2,2}^E(u_3,0) = J_0(2\pi\sqrt{E}|u_3|)+
J_2(2\pi\sqrt{E}|u_3|) = 2\frac{J_1(2\pi\sqrt{E}|u_3|)}{2\pi\sqrt{E}|u_3|}, 
\end{eqnarray*}
so that  changing  variable $v=2\pi\sqrt{E}u_3$,
\begin{eqnarray*}
\int_{-L}^{L} du_3 |\tilde{r}_{2,2}^E(u_3,0)| 
= \frac{1}{2\pi\sqrt{E}}
\int_{-2\pi\sqrt{E}L}^{2\pi\sqrt{E}L}2\frac{|J_1(|v|)|}{|v|} dv.
\end{eqnarray*}
Splitting the region of integrations and using that $|\tilde{r}^1_{2,2}(\cdot)|\leq 2$ yields
\begin{eqnarray*}
\frac{1}{2\pi\sqrt{E}L}
\int_{-2\pi\sqrt{E}L}^{2\pi\sqrt{E}}2\frac{|J_1(|v|)|}{|v|} dv
= \frac{1}{2\pi\sqrt{E}}
\int_{-1}^{1}2\frac{|J_1(|v|)|}{|v|} dv
+ \frac{2}{2\pi\sqrt{E}}\int_{1}^{2\pi\sqrt{E}L}2\frac{|J_1(v)|}{v} dv.
\end{eqnarray*}
The first term is $O(E^{-1/2})$. For the second term, we use the bound $|J_{1}(v)|\leq v^{-1/2}$, to obtain 
\begin{eqnarray*}
\frac{2}{2\pi\sqrt{E}}\int_{1}^{2\pi\sqrt{E}L}2\frac{|J_1(v)|}{v} dv
\leq \frac{2}{2\pi\sqrt{E}}\int_{1}^{2\pi\sqrt{E}L}\frac{1}{v^{3/2}} = O(E^{-1/2}). 
\end{eqnarray*}
For the integration with respect to $u_1$ in  \eqref{Eq:intU}, we have that 
\begin{eqnarray*}
\int_{-4L}^{4L} |r^E(u_1,0)| du_1 
&=& \frac{1}{2\pi\sqrt{E}}\(O(1)
+ 2\int_{1}^{8\pi\sqrt{E}L} |J_0(v)| dv\)\\
&\leq& 
\frac{1}{2\pi\sqrt{E}}\(O(1)
+ 2\int_{1}^{8\pi\sqrt{E}L} \frac{1}{\sqrt{v}}  dv\)
= O(E^{-1/4}).
\end{eqnarray*}
From this, we deduce that 
the integral in \eqref{integral} is $O(E\cdot E^{-1/2}E^{-1/2}E^{-1/4}) = O(E^{-1/4})$, which suffices.

\end{proof}

We now conclude the proofs of Theorem \ref{AsyVar} and Theorem \ref{CLTd}. 
\begin{proof}[Proof of Theorem \ref{AsyVar}]
Assume that $\CC_1,\CC_2\in \mathscr{C}$ are given by $\CC_1= (\CC_1, ... \CC_{r_1})$    
and $\CC_2=( T_1,... T_{r_2})$ for segments $S_1,\ldots,S_{r_1},T_1,\ldots, T_{r_2}$. Then, 
 in view of the covariance structure for line segments proved in Proposition \ref{PropCovSeg}, it follows that, as $E\to \infty$,
\begin{eqnarray*}
&&\Cov{\phi_E(\CC_1)}{\phi_E(\CC_2)}
= \sum_{j=1}^{r_1}\sum_{k=1}^{r_2} 
\Cov{\phi_E(S_j)}{\phi_E(T_k)}\\
&&=  \sum_{j=1}^{r_1}\sum_{k=1}^{r_2} 
\frac{\lambda(S_j,T_k)}{16\pi^2\sqrt{E}}
+ o\(\frac{1}{\sqrt{E}}\)
= \frac{\lambda(\CC_1,\CC_2)}{16\pi^2\sqrt{E}}+o\(\frac{1}{\sqrt{E}}\),
\end{eqnarray*}
 where we used the definition of $\lambda(\cdot,\cdot)$. The variance estimate follows after setting $\CC_1=\CC_2$ above. \end{proof}
\begin{proof}[Proof of Theorem \ref{CLTd}]
Thanks to the estimate in \eqref{Eq:AsyCov} and \cite[Theorem 6.2.3]{NP:12}, it suffices to prove the statement for  $d=1$.  
 Write $\CC= (S_1,...,S_r)$ for line segments $S_1,\ldots, S_r$. The definition of $\phi_E$ and the fact that the random vector $(\wt{\phi}_{E}(S_1),\ldots, \wt{\phi}_E(S_r))$ converges in distribution to a Gaussian  vector with covariance matrix 
 $\Sigma(i,j)=\lambda(S_i,S_j), i,j=1,\ldots,r$ (in view of Proposition \ref{PropCLTLS}) imply that $\wt{\phi}_E(\CC)$ converges in distribution to a Gaussian random variable with  variance $\lambda(\CC,\CC)$. 
\end{proof}

\subsection{Proof of Corollary \ref{Cor:SecondFCLT}}\label{Sec:SecondFCLT}
In view of the variance estimate in Theorem \ref{AsyVar}, and taking into account the normalization in the definition of $X_E$ (see \eqref{Eq:XE}), we deduce that the finite-dimensional distributions of $X_E[2]$ converge to zero. We are thus left to show that the laws of $\{X_E[2]:E>0\}$ are tight, which is the content of the following proposition. 

\begin{Prop}\label{PropTight2}
The laws of the random functions $\{X_E[2]:E>0\}$ are tight in $\mathbf{D}_2$.
\end{Prop}

The proof of Proposition \ref{PropTight2} is based on the following criterion by Davydov and Zitikis \cite[Theorem 1]{DZ:08} for proving weak convergence of processes on $[0,1]^d$.  

\begin{Thm}[see \cite{DZ:08}]\label{Thm:DZ1Ch}
Let $Y_n=\{Y_n(t):t\in [0,1]^d\}, n\geq 1$ be a collection of real-valued stochastic processes on $[0,1]^d$ such that its paths belong  $\Prob$-almost surely to $C([0,1]^d,\R)$. Assume furthermore that 
\begin{enumerate}[label=(\alph*)]
\item[\rm (a)] the finite-dimensional distributions of $Y_n$ converge to those of some stochastic process $Y$, 
\item[\rm (b)] there exist $\alpha\geq \beta>d, c>0$ and a numerical sequence $\{\alpha_n:n\geq 1\}$ such that $\alpha_n\to 0$ as $n\to \infty$, $\E{|Y_n(0)|^{\alpha}}\leq c $ for every $n\geq 1$ and 
\begin{eqnarray}\label{Eq:Cond1Ch}
\E{|Y_n(t)-Y_n(s)|^{\alpha}}
\leq c \norm{t-s}^{\beta}, \quad 
\forall t,s \in [0,1]^d: \norm{t-s}\geq \alpha_n,
\end{eqnarray}
\item[\rm (c)] for the sequence $\{\alpha_n:n\geq 1\}$ at point (b), we have as $n \to \infty$, 
\begin{eqnarray}\label{Eq:Cond2Ch}
\omega_{Y_n}(\alpha_n):= \sup_{\norm{t-s}\leq \alpha_n}
|Y_n(t)-Y_n(s)|
\toP 0.
\end{eqnarray} 
\end{enumerate}
Then, as $n\to \infty$, $Y_n$ converges weakly to $Y$ and $Y$ has continuous paths $\Prob$-almost surely.
\end{Thm} 
In order to prove that the process $X_E[2]$ verifies the assumptions \eqref{Eq:Cond1Ch} and \eqref{Eq:Cond2Ch}, our arguments make use of the following moment estimates for suprema of stationary Gaussian random fields. 
Here, for a function $f:\R^d\to\R$, a domain $\D\subset \R^d$ and an integer $j\geq 0$,  we denote by 
\begin{eqnarray*}
\norm{f}_{C^j(\D)} := \sup_{x\in \D} \sup_{|\alpha|\leq j} |\partial_{\alpha} f(x)|
\end{eqnarray*}
where $\partial_{\alpha}f(x):= \partial^{\alpha_1}_{x_1}\ldots \partial^{\alpha_d}_{x_d}f(x)$, for $\alpha:=(\alpha_1,\ldots, \alpha_d) $ with $|\alpha| := \sum_{k=1}^d \alpha_k$.  
 The proof of Proposition \ref{MomChapter} is postponed to  Appendix \ref{App:SupGauss}.
\begin{Prop} \label{MomChapter}
Let $G$ be a stationary Gaussian random field on $\R^d$. Assume that for every $m\geq 0$, there exists a constant $\tilde{\sigma}^2(m)<\infty$ such that   
\begin{eqnarray}\label{AssChapter}
\E{(\partial_{\alpha} G(x))^2} \leq \tilde{\sigma}^2(m), \quad \forall \alpha \in \N^d, |\alpha|\leq m.
\end{eqnarray}
Then, for any $p\geq 1$, \begin{eqnarray*}
\E{\norm{G}_{C^j(\D )}^p} \leq C \{\log (\mathrm{vol}(\D))\}^{p/2}
\end{eqnarray*}
where $C>0$ is an absolute constant depending on $p$ and $j$, and $\mathrm{vol}(\D)$ is the volume of $\D$.
\end{Prop}

The following auxiliary results are needed to complete the proof Proposition \ref{PropTight2}. 
\begin{Lem}\label{Contf}
For $\bb{t}=(t_1,t_2)\in [0,1]^2$, we write $\D_{\bb{t}}:=[0,t_1]\times [0,t_2]$.
Then, for every continuous function $f:[0,1]^2\to \R$ and every $\bb{t},\bb{s}\in [0,1]^2$, we have
\begin{eqnarray*}
\left| \int_{\mathcal{D}_{\bb{t}}} f(x) dx 
- \int_{\mathcal{D}_{\bb{s}}} f(x) dx \right| \leq C \sup_{x\in [0,1]^2} |f(x)|\norm{\bb{t}-\bb{s}},
\end{eqnarray*}
for some absolute constant $C>0$.
\end{Lem}
\begin{proof}
This follows directly from the fact that $\mathrm{area}(\D_{\bb{t}}\setminus \D_{\bb{s}}) \leq C \norm{\bb{t}-\bb{s}}$, for some constant $C>0$ which is independent of $\bb{t}$ and $\bb{s}$. 
\end{proof} 
 
\begin{Lem}\label{SupGaussian}
For every $p\geq 1$ and $E>0$, we have  
\begin{eqnarray*}
\E{\sup_{x \in [0,1]^2} |B_E(x)|^p}
+ \E{ \sup_{x\in [0,1]^2} \norm{\wt{\nabla} B_E(x)}^p} \leq C (\log E)^{p/2},
\end{eqnarray*}
where $C>0$ is some absolute constant depending only on $p$. 
\end{Lem}  
\begin{proof}
We use the fact that $B_E \eqLaw B_1(2\pi\sqrt{E}\cdot)$ as random fields, so that 
\begin{eqnarray*}
\E{\sup_{x\in [0,1]^2} B_E(x)^p}
= \E{\sup_{y\in [0,2\pi\sqrt{E}]^2} B_1(y)^p}.
\end{eqnarray*}
It is easy to see that the assumption   \eqref{AssChapter} of Proposition \ref{MomChapter} is satisfied by $B_1$. Applying the estimate in Proposition \ref{MomChapter} with $\D=[0,2\pi\sqrt{E}]^2 \subset \R^2$ yields the desired conclusion. The second supremum involving the normalized gradient is dealt with in the same way.
\end{proof}

We are now in the position to prove Proposition \ref{PropTight2}.
\begin{proof}[Proof of Proposition \ref{PropTight2}]
In view of Theorem \ref{Thm:DZ1Ch} and the fact that the finite-dimensional distributions of $X_E[2]$ converge to those of the   zero-process, it is sufficient to prove that there exists a numerical sequence $\{a_E:E>0\}$ such that $a_E \to 0$ as $E\to \infty$ and 
(i) there exist absolute constants $\alpha\geq \beta>2, c>0$ such that 
\begin{eqnarray}\label{i}
\E{ \left| X_E[2](\bb{t}) -X_E[2](\bb{s}) \right|^{\alpha} } \leq c \norm{\bb{t}-\bb{s}}^{\beta} , \quad 
\forall\bb{t},\bb{s}:\norm{\bb{t}-\bb{s}} \geq a_E
\end{eqnarray}
 and (ii)
\begin{eqnarray}\label{ii}
\omega(E):= \sup_{\norm{\bb{t}-\bb{s}} \leq a_E} \left|X_E[2](\bb{t}) -X_E[2](\bb{s}) \right|
\toP 0,
\end{eqnarray}
as $E\to \infty$.
We claim that choosing $a_E:= (\sqrt{E}\log E)^{-1}$ verifies the above conditions (i) and (ii). 
Let us prove that (i) holds. The variance estimate in Theorem \ref{AsyVar} implies that there exists an absolute constant $K>0$ such that for every $E>0$ and every $\bb{t}\in [0,1]^2$, $\V{X_E[2](\bb{t})}\leq K( \sqrt{E} \log E)^{-1}$. Therefore, choosing $\alpha=2$ in \eqref{i}, we infer that 
for every $\bb{t},\bb{s}$ such that $\norm{\bb{t}-\bb{s}}\geq (\sqrt{E}\log E)^{-1}$, 
\begin{gather*}
\E{ \left( X_E[2](\bb{t}) - X_E[2](\bb{s}) \right)^2 }
\leq 2 \V{X_E[2](\bb{t})} + 2\V{X_E[2](\bb{s})} 
\leq
\frac{2K}{\sqrt{E} \log E} \leq c\norm{\bb{t}-\bb{s}}. 
\end{gather*}
Since for every $\bb{t}\in [0,1]^2$, $X_E[2](\bb{t})$ is an element of the second Wiener chaos associated with $B_E$, we exploit the hypercontractivity property for multiple Wiener integrals (see \cite[Theorem 2.7.2]{NP:12}) to obtain (for some absolute constant $C>0$) \begin{eqnarray*}
 \E{ \left( X_E[2](\bb{t}) - X_E[2](\bb{s}) \right)^p } 
\leq C\E{ \left( X_E[2](\bb{t}) - X_E[2](\bb{s}) \right)^2 }^{p/2}  
\leq C\norm{\bb{t}-\bb{s}}^{p/2}  
\end{eqnarray*}
for every $p>4$, 
which gives the desired estimate in \eqref{i} since $p/2>2$. 
Let us now argue that (ii) holds. 
By \cite[Eq. (4.58)]{NPR:19}, we can write  
\begin{eqnarray*}\label{dual}
X_E[2](\bb{t}) &=& \sqrt{\frac{512\pi}{\log E}} \mathscr{L}_E[2](\mathcal{D}_{\bb{t}}) =
\sqrt{\frac{512\pi}{\log E}} \frac{\pi}{8} \sqrt{2E}
\left[-2\int_{\mathcal{D}_{\bb{t}}} B_E(x)^2 dx +
\int_{\mathcal{D}_{\bb{t}}} \norm{\wt{\nabla} B_E(x)}^2 dx\right] \notag \\ 
&=& 4\pi^{3/2} \sqrt{\frac{E}{\log E}} \left[-2\int_{\mathcal{D}_{\bb{t}}} B_E(x)^2 dx +
\int_{\mathcal{D}_{\bb{t}}} \norm{\wt{\nabla} B_E(x)}^2 dx\right].
\end{eqnarray*}
Combining this expression with
Lemma \ref{Contf} applied to $f=B_E(\cdot)^2$ and $f= \norm{\wt{\nabla} B_E(\cdot)}^2$,  yields for every choice of $\bb{t},\bb{s}$ such that $\norm{\bb{t}-\bb{s}} \leq a_E$  (denoting by $C$ an absolute constant whose value varies from line to line)
\begin{eqnarray*}
&&\left|X_E[2](\bb{t}) -X_E[2](\bb{s}) \right| \\
&\leq& C \sqrt{\frac{E}{\log E}}
\{ \left| \int_{\mathcal{D}_{\bb{t}}} B_E(x)^2dx- \int_{\mathcal{D}_{\bb{s}}} B_E(x)^2 dx\right| 
+ \left| \int_{\mathcal{D}_{\bb{t}}}  \norm{\wt{\nabla} B_E(x)}^2dx 
- \int_{\mathcal{D}_{\bb{s}}}  \norm{\wt{\nabla} B_E(x)}^2\right|
 \}
\\
&\leq& C\sqrt{\frac{E}{\log E}}  \{ \sup_{x\in [0,1]^2} |B_E(x)|^2 
+ \sup_{x\in [0,1]^2}\norm{\wt{\nabla} B_E(x)}^2    \} \norm{\bb{t}-\bb{s}}  \\
&\leq & \frac{C}{(\log E)^{3/2}}\{ \sup_{x\in [0,1]^2} |B_E(x)|^2 
+ \sup_{x\in [0,1]^2}\norm{\wt{\nabla} B_E(x)}^2    \}  ,
\end{eqnarray*}
where we used the definition of $a_E$.
This implies that 
\begin{eqnarray*}
\E{\omega(E)} &=& 
\E{  \sup_{\norm{\bb{t}-\bb{s}} \leq a_E} \left|X_E[2](\bb{t}) -X_E[2](\bb{s}) \right| } \\
&\leq &\frac{C}{(\log E)^{3/2}}\{ \E{\sup_{x\in [0,1]^2} |B_E(x)|^2 }
+ \E{\sup_{x\in [0,1]^2}\norm{\wt{\nabla} B_E(x)}^2 dx}  \}  \\
&\leq& \frac{C}{(\log E)^{3/2}} \cdot \log E = \frac{C}{\sqrt{\log E}}
\end{eqnarray*}
where we used Lemma \ref{SupGaussian} with $p=2$. 
Therefore, by the Markov inequality we have for every $\eta>0$,
\begin{eqnarray*}
\Prob\{ \omega(E) >\eta \} \leq 
\eta^{-1} \E{\omega(E)} \leq \frac{C}{\eta \sqrt{\log E}},
\end{eqnarray*}
which proves the validity of (ii). 
\end{proof}

\subsection{Proof of Theorem \ref{Thm:DiscretizedR}
and Corollary \ref{Cor:Discretized}
}\label{Sec:DiscretizedR}
Our proof of Theorem \ref{Thm:DiscretizedR} is based on a planar chaining argument, similar to the one presented in \cite{DT:89} and \cite{MW:11b} in dimension one for a study related to empirical processes. 

\medskip
We start with a preliminary lemma, yielding a $L^2$ bound for increments of $R_E= \sum_{q\geq 3}X_E[2q]$ along rectangles of the form $[s_1,t_1]\times [s_2,t_2]\subset [0,1]^2$.
\begin{Lem}\label{Cheby}
For every $0\leq s_i\leq t_i \in [0,1], i=1,2$, and $E>0$,  we have that 
\begin{eqnarray*}
\E{\(R_E(t_1,t_2)-R_E(s_1,t_2)
- R_E(t_1,s_2)
+R_E(s_1,s_2)\)^2}\leq  
 \frac{C}{\log E}\[(t_1-s_1)(t_2-s_2)\],
\end{eqnarray*}
where $C>0$ is some absolute constant (independent of $t_1,t_2,s_1,s_2$ and $E$).
\end{Lem}
\begin{proof}
Let $\D(\bb{t},\bb{s}):=[s_1,t_1]\times [s_2,t_2]$.
By definition of $R_E$ and the additivity of the nodal length, we have 
\begin{eqnarray}\label{diff}
R_E(t_1,t_2)-R_E(s_1,t_2)
- R_E(t_1,s_2)
+R_E(s_1,s_2)  
=\sqrt{\frac{512\pi}{\log E}}\sum_{q\geq3}\mathscr{L}_E{[2q]}(\D(\bb{t},\bb{s})) .
\end{eqnarray} 
By inspection of the arguments used in the proofs of \cite[Lemmas 7.6,7.8,7.9]{NPR:19}, one verifies that there is an absolute constant $C_1>0$ (independent of $\bb{t},\bb{s}$ and $E$) such that
\begin{eqnarray*}
\V{\sum_{q\geq3}\mathscr{L}_E[2q](\D(\bb{t},\bb{s}))} \leq C_1 \mathrm{area}(\D(\bb{t},\bb{s})).
\end{eqnarray*}
Taking the square of $L^2(\Prob)$-norm in \eqref{diff}, and exploiting the above upper bound, we obtain
\begin{gather*}
\E{\(R_E(t_1,t_2)-R_E(s_1,t_2)
- R_E(t_1,s_2)
+R_E(s_1,s_2)\)^2}
=
\frac{512\pi}{\log E}\V{\sum_{q\geq3}\mathscr{L}_E[2q](\D(\bb{t},\bb{s}))}\\
\leq C_1 \frac{512\pi}{\log E}\mathrm{area}(\D(\bb{t},\bb{s}))
=  \frac{C}{\log E} (t_1-s_1)(t_2-s_2), 
\end{gather*}
which gives the desired conclusion.
\end{proof}

We are now in the position to prove Proposition \ref{Thm:DiscretizedR}. 
\begin{proof}[Proof of Theorem \ref{Thm:DiscretizedR}]
We start by introducing refining partitions of the unit square. 

\medskip
\noindent{\bf Refining partitions of the unit square.}
  Let us fix a large integer $K$ (whose exact value will be chosen later as a function of $E$).
For every integers $k,k'=0,\ldots, K$, and every vector $i =(i_1,i_2)\in \{0,\ldots, 2^{k}\}\times \{0,\ldots, 2^{k'}\}$, we define the {\bf partition points} 
\begin{eqnarray*}
\bb{p}_{i}(k,k'):=\(p_{i_1}(k),p_{i_2}(k')\)
= \(\frac{i_1}{2^k}, \frac{i_2}{2^{k'}}\) \in [0,1]^2.
\end{eqnarray*}
Moreover, for every $\bb{t}=(t_1,t_2)\in [0,1]^2$ and   $k,k'=0,\ldots, K$, we define the vector $i_{k,k'}(\bb{t})=\(i_{1,k}(t_1),i_{2,k'}(t_2)\)$ to be such that  
\begin{eqnarray*}
p_{i_{1,k}(t_1)}(k)\leq t_1 \leq p_{i_{1,k}(t_1)+1}(k),  \quad 
p_{i_{2,k'}(t_2)}(k')\leq t_2 \leq p_{i_{2,k'}(t_2)+1}(k'),
\end{eqnarray*}
 that is, for every $\bb{t}\in [0,1]^2$, the vector $i_{k,k'}(\bb{t})$ is such that $\bb{p}_{i_{k,k'}(\bb{t})}(k,k')$ is the closest partition point to $\bb{t}$ on the left. 
 
 \medskip
We introduce the following operators. 
\begin{Def} 
Given a function $f:[0,1]^2 \to \R$ a point $\bb{t}=(t_1,t_2) \in [0,1]^2$,  and $k,k'\in \{0,\ldots,K-1\}$, we define the {\bf difference operator}
\begin{eqnarray*}
\Delta_{k,k'}f(\bb{t})
&:=& f\( p_{i_{1,k+1}(t_1)}(k+1), p_{i_{2,k'+1}(t_2)}(k'+1) \)
- f\( p_{i_{1,k+1}(t_1)}(k+1), p_{i_{2,k'}(t_2)}(k') \)\\
&&- f\( p_{i_{1,k}(t_1)}(k), p_{i_{2,k'+1}(t_2)}(k'+1) \)
+ f\( p_{i_{1,k}(t_1)}(k), p_{i_{2,k'}(t_2)}(k') \).
\end{eqnarray*}
Also, for $k,k' \in \{0,\ldots, K-1\}$, we set 
 \begin{eqnarray*}
\Delta_{K,k'}f(\bb{t})
&:=& f\( t_1, p_{i_{2,k'+1}(t_2)}(k'+1) )
- f( t_1, p_{i_{2,k'}(t_2)}(k') \)\\
&&- f\( p_{i_{1,K}(t_1)}(K), p_{i_{2,k'+1}(t_2)}(k'+1) \)
+ f\( p_{i_{1,K}(t_1)}(K), p_{i_{2,k'}(t_2)}(k') \),\\
\Delta_{k,K}f(\bb{t})
&:=&f\( p_{i_{1,k+1}(t_1)}(k+1), t_2\)
- f\( p_{i_{1,k+1}(t_1)}(k+1), p_{i_{2,K}(t_2)}(K) \)\\
&&- f\( p_{i_{1,k}(t_1)}(k), t_2 \)
+ f\( p_{i_{1,k}(t_1)}(k), p_{i_{2,K}(t_2)}(K) \)
\end{eqnarray*}
and finally
\begin{eqnarray*}
\Delta_{K,K}f(\bb{t})
&:=&f\( t_1, t_2\)
- f\(t_1, p_{i_{2,K}(t_2)}(K) \)\\
&&- f\( p_{i_{1,K}(t_1)}(K), t_2 \)
+ f\( p_{i_{1,K}(t_1)}(K), p_{i_{2,K}(t_2)}(K) \).
\end{eqnarray*}
Also, we use the notations $\Delta_{k,k'}^+, \Delta_{K,k'}^+$ and 
$\Delta_{k,K}^+$ to indicate the   operators obtained from the relations above by replacing $t_1$ and $t_2$ with $p_{i_{1,k}(t_1)+1}(k)$ and 
$p_{i_{2,k'}(t_2)+1}(k')$, respectively.
 \end{Def}
 
We remark that, by construction of the refining partitions, we have either   $p_{i_{1,k}}(k) = p_{i_{1,k+1}(t_1)}(k+1)$ or 
$p_{i_{1,k+1}(t_1)}(k+1)-p_{i_{1,k}(t_1)}(k) = 2^{-(k+1)}$ (and similarly for partition coordinates involving the index $i_2$) which yields in particular 
 \begin{eqnarray}\label{Eq:Delta}
&&\left| \Delta_{k,k'} f(t_1,t_2)\right|\notag\\
&\leq&
\biggl| f\(p_{i_{1,k}(t_1)}(k)+\frac{1}{2^{k+1}}, 
p_{i_{2,k'}(t_2)}(k')+\frac{1}{2^{k'+1}}\)
-   f\(p_{i_{1,k}(t_1)}(k)+\frac{1}{2^{k+1}}, 
p_{i_{2,k'}(t_2)}(k')\)\notag\\
&&- f\(p_{i_{1,k}(t_1)}(k), 
p_{i_{2,k'}(t_2)}(k')+\frac{1}{2^{k'+1}}\)
+ f\(p_{i_{1,k}(t_1)}(k), 
p_{i_{2,k'}(t_2)}(k')\)\biggr|.
 \end{eqnarray}
In view of the above defined difference operators, the following bivariate telescopic formula holds
\begin{eqnarray}\label{Eq:Tele}
f(t_1,t_2) = \sum_{k,k'=0}^K
\Delta_{k,k'}f(t_1,t_2)
\end{eqnarray}
for every $f:[0,1]^2\to \R$. 

\medskip
Let us now write $R_E^K$ for the discretized version of $R_E$ associated with the above partition. Applying \eqref{Eq:Tele} to  $R_E^K$, we can write for every $\bb{t}\in [0,1]^2$,
\begin{eqnarray}\label{Eq:BBc}
\left|R_E^K(\bb{t})\right|
&=& \left|\sum_{k,k'=0}^K
\Delta_{k,k'}R_E^K(\bb{t})\right|\notag \\
&\leq&  \sum_{(k,k')\in B(K)^c}
\left|\Delta_{k,k'}R_E^K(\bb{t})\right|
+ \left|\sum_{(k,k')\in B(K)}
\Delta_{k,k'}R_E^K(\bb{t})\right|,
\end{eqnarray}
where we set $B(K):=
\{(k,k')\in \{0,\ldots, K\}^2: \max(k,k')=K\}$. Note that the second term in the R.H.S of \eqref{Eq:BBc} vanishes by definition of the operators $\Delta_{K,k'},\Delta_{k,K}$ and $\Delta_{K,K}$, and the fact that we consider the discretized remainder $R_E^K$: indeed, for every $(k,k') \in B(K)$, we have that 
$\Delta_{k,k'}R_E^K(\bb{t}) = 0$.
From this, we conclude that, for every $\eps >0$, 
\begin{eqnarray}\label{Eq:BoundProb}
\Prob\{ \sup_{\bb{t}\in [0,1]^2}
\left|R_E^K(\bb{t})\right|>\eps\} 
\leq \Prob\{ \sup_{\bb{t}\in [0,1]^2}
\sum_{(k,k')\in B(K)^c}
\left|\Delta_{k,k'}R_E^K(\bb{t})\right|>\eps\}. 
\end{eqnarray}
We remark that the R.H.S involves the increments on closest partition points associated with $\bb{t}$. 
Now, using the fact that 
\begin{eqnarray*}
\sum_{k,k'\geq 0} \frac{\eps}{(k+3)^2(k'+3)^2} \leq \eps
\end{eqnarray*}
and the Chebychev inequality, we can bound the probability in \eqref{Eq:BoundProb} by
\begin{eqnarray*}
&& \sum_{(k,k')\in B(K)^c}
\Prob\{ \sup_{\bb{t}\in [0,1]^2}
\left|\Delta_{k,k'}R_E^K(\bb{t})\right|> \frac{\eps}{(k+3)^2(k'+3)^2}
\}\\
&\leq& \sum_{(k,k')\in B(K)^c} 
\sum_{i_1=0}^{2^k}\sum_{i_2=0}^{2^{k'}}
\Prob\{ 
\left|\Delta_{k,k'}R_E^K(\bb{p}_{(i_1,i_2)}(k+1,k+1))\right|> \frac{\eps}{(k+3)^2(k'+3)^2}\}\\
&\leq&\sum_{(k,k')\in B(K)^c} \frac{(k+3)^4(k'+3)^4}{\eps^2}
\sum_{i_1=0}^{2^k}\sum_{i_2=0}^{2^{k'}} \frac{C}{\log E} \frac{1}{2^{k+1}}\frac{1}{2^{k'+1}} \leq \frac{C'}{\log E} K^{10},
\end{eqnarray*}
where we used Lemma \ref{Cheby}. Therefore, this probability converges to zero once we chose $K=K(E)$ in such a way that 
$K(E)\to \infty $ and $K(E) = o((\log E)^{1/10})$ as $E\to\infty$. This concludes the proof. 
\end{proof}

\begin{proof}[Proof of Corollary \ref{Cor:Discretized}]
Let us choose $K=K(E)$ as in the statement. 
By the Wiener chaos expansion of $X_E^{K}$, we can write 
\begin{gather*}
X_E^K 
= X_E[4]  + \(X_E^K[4] 
- X_E[4]  \)
+X_E[2] + \(X_E^K[2] 
- X_E[2] \)
+ R_E^K .
\end{gather*}
We use the same strategy used to prove Lemma \ref{Lem:UVW}. The process $X_E[4]$ converges weakly to a standard Wiener sheet in the space $\mathbf{D}_2$, in view  of \cite[Theorem 3.4]{PV:20}. The residue process $R_E^K$ converges to zero uniformly in probability, in view of Theorem \ref{Thm:DiscretizedR}.
The second chaotic projections converge weakly to zero in view of Corollary \ref{Cor:SecondFCLT}.
For the term $X_E^K[4]-X_E[4]$ we argue that, for every $\eps>0$,
\begin{eqnarray*}
\Prob\{ \sup_{\bb{t}\in [0,1]^2 } \left|X_E^K[4](\bb{t})-X_E[4](\bb{t})\right| > \eps \} \to 0
\end{eqnarray*}
as $E\to \infty$. 
By definition of $X_E^K[4]$, we can rewrite   
\begin{eqnarray*}
\Prob\{ \sup_{\bb{t}\in [0,1]^2 } \left|X_E^K[4](\bb{t})-X_E[4](\bb{t})\right| > \eps \} 
=\Prob\{ \sup_{\bb{t}\in [0,1]^2 } \left|X_E[4](\bb{p}_{i_{K,K}(\bb{t})}(K,K))-X_E[4](\bb{t})\right| > \eps \} .
\end{eqnarray*}
Since both $X_E[4]$ and $\bb{W}$ belong to the space $\mathbf{C}_2$, $\Prob$-almost surely, we can apply  the Skorohod Representation Theorem   \cite[Theorem 11.7.2]{Du:02}. Thus, on some probability space $(\Omega',\mathscr{F}',\Prob')$, there exist  $\{Y_E':E>0\}, Z' \in \mathbf{C}_2$  such that $Y_E' \eqLaw X_E[4] , Z' \eqLaw \mathbf{W}$ and 
$$\sup_{\mathbf{t}\in [0,1]^2} | Y_E'(\mathbf{t})-Z'(\mathbf{t})| \to 0,$$ $\Prob'$-almost surely as $E\to \infty$.  
Therefore, denoting by $\Delta\in P_E$ a cell of a partition $P_E$ of $[0,1]^2$ with mesh $|P_E|\to 0$ as $E\to\infty$, we can write 
\begin{eqnarray*}
&&\Prob\{ \sup_{\bb{t}\in [0,1]^2 } \left|X_E[4](\bb{p}_{i_{K,K}(\bb{t})}(K,K))-X_E[4](\bb{t})\right| > \eps \}\\
&\leq&
\Prob'\{\sup_{\Delta\in P_E}
\sup_{\mathbf{t},\mathbf{s}\in \Delta}
|{X_E' (\bb{t})-X_E'(\bb{s})}| > \eps\} \\
&\leq & \Prob'\{\sup_{\Delta\in P_E}
\sup_{\mathbf{t},\mathbf{s}\in \Delta}
\(|{X_E' (\bb{t})-Z'(\bb{t})}| + 
|{Z' (\bb{s})-X_E'(\bb{s})}|\)  > \frac{\eps}{2}\}\\
&& \qquad
+  \Prob'\{\sup_{\Delta\in P_E}
\sup_{\mathbf{t},\mathbf{s}\in \Delta} 
|{Z' (\bb{t})-Z'(\bb{s})}|  > \frac{\eps}{2}\}
\end{eqnarray*}
where we used the triangle inequality. 
For the first probability, we use the fact that $X'_E$ converges   to $ Z'$   uniformly on $[0,1]^2$, $\Prob'$-almost surely, implying that  
 \begin{eqnarray*}
 \sup_{\Delta\in P_E}
\sup_{\mathbf{t},\mathbf{s}\in \Delta}
\(|{X_E' (\bb{t})-Z'(\bb{t})}| + 
|{Z' (\bb{s})-X_E'(\bb{s})}|\)
\leq 
2\sup_{\mathbf{t}\in [0,1]^2}
|X_E'(\bb{t})-Z'(\bb{t})| \to 0,
\end{eqnarray*} 
$\Prob'$-almost surely. For the second term, we notice that $Z'$ is uniformly continuous on $[0,1]^2$ (being continuous on $[0,1]^2$), so that $\sup_{\bb{t},\bb{s}\in \Delta} |Z'(\bb{t})-Z'(\bb{s})| \to 0$, $\Prob'$-almost surely.
\end{proof}

\subsection{Proof of Proposition \ref{Prop:TruncatedR}}
By Lemma \ref{Cheby}, we deduce that for every $\bb{t}\in [0,1]^2$, 
\begin{eqnarray*}
\V{R_E(\bb{t})} = O\(\frac{1}{\log E}\),
\end{eqnarray*}
where the constants involved in the 'big-O' notation are independent of $\bb{t}$ and $E$. This implies that 
 the finite-dimensional distributions of 
the process $R_E(\bullet,N)$ converge to zero for every fixed $N\geq 4$. Therefore, in order to obtain the desired conclusion, it is sufficient to prove that the laws of the random mappings
$\{R_E(\bullet; N):E>0\}$ (for $N=N(E)$ as in the statement) verify a Kolmogorov type estimate of the form
\begin{eqnarray}\label{Eq:KOL}
\E{\( R_E(\bb{t};N) - R_E(\bb{s};N)\)^{\alpha}}
\leq c \norm{\bb{t}-\bb{s}}^{2+\beta}
, \quad \forall \bb{t},\bb{s}\in [0,1]^2
\end{eqnarray}
for some constants $\alpha,\beta>0$ and $c>0$ that are independent of $E$. 
Denoting by $\D_{\bb{t}}:= [0,t_1]\times [0,t_2]$ and $\D_{\bb{t},\bb{s}}:= \D_{\bb{t}}\setminus \D_{\bb{s}}$, we have that for every integer $N$ (to be chosen later as a function of $E$)  and every $p>2$,
\begin{eqnarray*}
\E{\( R_E(\bb{t};N) - R_E(\bb{s};N)\)^{p}}^{1/p}
= \frac{512\pi}{\log E} \E{\( \sum_{q=3}^N \mathscr{L}_E[2q](\D_{\bb{t},\bb{s}})\)^{p}}^{1/p}.
\end{eqnarray*}
Since $\sum_{q=3}^N \mathscr{L}_E[2q](\D_{\bb{t},\bb{s}})$ is a random variable living in the orthogonal sum of Wiener chaoses up to order $2N$, we use the hypercontractivity property (\cite[Theorem 2.7.2]{NP:12}) together with 
Lemma \ref{Cheby}, to deduce that 
\begin{eqnarray*}
\E{\( R_E(\bb{t};N) - R_E(\bb{s};N)\)^{p}}^{1/p}
&\leq& \frac{512\pi}{\log E} 
(p-1)^{N} 
\V{ \sum_{q=3}^N \mathscr{L}_E[2q](\D_{\bb{t},\bb{s}}) }^{1/2}\\
&\leq&  \frac{c_p}{ \log E } (p-1)^{N} \norm{t-s}^{1/2},
\end{eqnarray*}
where $c_p$ is some absolute constant only depending on $p$.
In particular, for $p=6$ we obtain the estimate 
\begin{eqnarray*}
\E{\( R_E(\bb{t};N) - R_E(\bb{s};N)\)^{6}}
\leq \{\frac{c_6}{ \log E } 5^{N} \norm{t-s}^{1/2}\}^6 = c \frac{5^{6N}}{(\log E)^6}
\norm{t-s}^3,
\end{eqnarray*}
for some absolute constant $c>0$. Thus in order to prove \eqref{Eq:KOL}, it is sufficient to choose $N=N(E)$ such that (the constant $1$ is not important here) 
\begin{eqnarray*}
\frac{5^{6N(E)}}{(\log E)^6} = 1,
\end{eqnarray*}
yielding that $N(E)=6^{-1} \log_5( (\log E)^6)=\log_5(\log E)$. This proves the claim.

\begin{appendix}

\section{Proof of Proposition \ref{p:whitenoise}}\label{ss:proofwhitenoise}

According to Proposition \ref{p:fdd}, the random field $X_E$ converges in the sense of finite-dimensional distributions to ${\bf W}$, and moreover one has that
$$
\sup_{E>0, \, {\bf t}\in [0,1]^2} \E{ X_E({\bf t})^2} <\infty, \quad \mbox{and} \quad \E{ X_E({\bf t})^2}\to \E{ {\bf W}({\bf t})^2},
$$
where the first relation follows from the computations contained in \cite[Sections 6 and 7]{NPR:19}, and the second one takes place as $E\to\infty$, for all ${\bf t}\in [0,1]^2$. As a consequence of these relations, we can apply \cite[Theorem 4]{I:80} and conclude that, if $\psi \in \mathcal{C}_c^\infty(R)$, then 
\begin{equation}\label{e:ca}
\int_R X_E({\bf t}) \psi({\bf t}) d{\bf t} \Law \int_R {\bf W}({\bf t}) \psi({\bf t}) d{\bf t}.
\end{equation}
We now fix $\varphi\in \mathcal{C}_c^\infty(R)$ and apply \eqref{e:ca} to $\psi({\bf t}) = \frac{\partial}{\partial t_1}\frac{\partial}{\partial t_2}\varphi({\bf t})\in \mathcal{C}_c^\infty(R)$, where ${\bf t} = (t_1,t_2)$, in such a way that $ \varphi({\bf t}) = \int_{(t_1, 1)\times (t_2, 1)} \psi({\bf z}) d{\bf z}$. Applying a standard Fubini theorem on the left-hand side of \eqref{e:ca} and a stochastic Fubini theorem (see \cite[Theorem 5.13.1]{PT:10}) on the right-hand side yields that
$$
\int_R X_E({\bf t}) \psi({\bf t}) d{\bf t} = \langle \wt{\mathscr{L}}_E, \varphi\rangle, \quad \mbox{and} \quad \int_R {\bf W}({\bf t}) \psi({\bf t}) d{\bf t} = \int_R \varphi({\bf z}) {\bf W}(d{\bf z}),
$$
where the last expression denotes a stochastic Wiener-It\^o integral with respect to ${\bf W}$. The conclusion now follows from \cite[Theorem III.6.5]{F:67}. 

\section{Proof of Lemma \ref{Lem:UVW}}\label{App:UVW}
Since
$U_n$ and $V_n$ converge weakly to $X$ and zero in $\mathbf{D}_2$, respectively, we use  for instance \cite[Theorem 2]{Wi:69}, to    deduce   that, for every $\eps>0$,
\begin{eqnarray}\label{Eq:UV}
\lim_{\delta \to 0} \limsup_{n\to \infty} \Prob\{ \omega_{\delta}(U_n)>\eps\} = 0,
\quad 
\lim_{\delta \to 0} \limsup_{n\to \infty} \Prob\{ \omega_{\delta}(V_n)>\eps\} = 0
\end{eqnarray}
where $\omega_{\delta}(f) := \sup\{| f(\bb{t}) - f(\bb{s})| : \norm{\bb{t}-\bb{s}}<\delta\}$. To obtain the desired conclusion, by virtue of the discussion contained in \cite[p.1291]{Ne:71}, it is sufficient to show that 
\begin{eqnarray*}
\lim_{\delta \to 0} \limsup_{n\to \infty} \Prob\{ \omega_{\delta}(X_n)>\eps\} = 0.
\end{eqnarray*}
By the triangle inequality, we can write for every $\delta>0$,
\begin{eqnarray*}
\omega_{\delta}(X_n) = \omega_{\delta}(U_n+V_n+W_n) \leq \omega_{\delta}(U_n) +\omega_{\delta}(V_n)+ \omega_{\delta}(W_n),  
\end{eqnarray*}
in such a way that 
\begin{eqnarray*}
\Prob\{ \omega_{\delta}(X_n)>\eps\}
\leq \Prob\{  \omega_{\delta}(U_n)>\eps/3\}
+ \Prob\{ \omega_{\delta}(V_n)>\eps/3\}
+ \Prob\{ \omega_{\delta}(W_n)>\eps/3\}.
\end{eqnarray*}
Using the  estimate $\omega_{\delta}(W_n) \leq 2\sup_{\bb{t}\in [0,1]^2} |W_n(\bb{t})|$ and letting $n\to \infty$ and $\delta\to 0$ then implies the desired conclusion from  \eqref{Eq:UV} and assumption (iii) in the statement.

\section{Moment estimates for suprema of Gaussian fields}\label{App:SupGauss} 
In what follows we consider a centred smooth stationary Gaussian field $G=\{ G(x):x\in \R^d\}  $ on $\R^d$ with covariance function $\E{G(x)G(y)}=\kappa(x-y)$.  For an integer $j\geq 0$ and $\D \subset \R^d$, we write 
\begin{eqnarray*}
\sigma^2(\D;j):= \sup_{x\in \D} \sup_{|\alpha|\leq j} \E{ (\partial_{\alpha} G(x))^2},
\end{eqnarray*}
where $\partial_{\alpha}G(x):= \partial^{\alpha_1}_{x_1}\ldots \partial^{\alpha_d}_{x_d}G(x)$, for $\alpha:=(\alpha_1,\ldots, \alpha_d) $ with $|\alpha| := \sum_{k=1}^d \alpha_k$.
Moreover, for $\D \subset \R^d$ and $\eps>0$, we write $\D^{(\eps)}$ for the $\eps$-enlargement of $\D$. Finally, we use the notation
\begin{eqnarray*}
\norm{f}_{C^j(\D)} := \sup_{x\in \D} \sup_{|\alpha|\leq j} |\partial_{\alpha} f(x)|
\end{eqnarray*} 
for $f:\R^d\to \R$.   
The goal of this section is to prove Proposition \ref{MomChapter}, whose statement we recall for convenience.  
\begin{Prop} \label{Mom}
Let the above setting prevail. 
Assume that for every $m\geq 0$, there exists $\tilde{\sigma}^2(m)<\infty$ such that   
\begin{eqnarray}\label{Ass}
\E{(\partial_{\alpha} G(x))^2} \leq \tilde{\sigma}^2(m), \quad \forall \alpha \in \N^d, |\alpha|\leq m.
\end{eqnarray}
Then, for every $p\geq 1$ and $j\geq0$ 
\begin{eqnarray*}
\E{\norm{G}_{C^j(\D )}^p} \leq C \{\log (\mathrm{vol}(\D))\}^{p/2}
\end{eqnarray*}
where $C>0$ is an absolute constant depending on $p$ and $j$, and $\mathrm{vol}(\D)$  is the $d$-dimensional volume of $\D$.
\end{Prop}
We remark that 
  assumption \eqref{Ass} in particular implies that $\sigma^2(\D;j) \leq \tilde{\sigma}^2(j)$ for every $j\geq 0$.
  
\subsection{Proof of Proposition \ref{Mom}}  
The proof of Proposition \ref{Mom} is based on several classical concentration inequalities for suprema of Gaussian fields, that we state here below. 
The first statement  is an estimate for the first moment of $\norm{G}_{C^j(\D)}$ (see \cite[Appendix A.9]{NS:16}).
\begin{Prop}\label{PropNS}
Let the above setting prevail. 
\begin{eqnarray}\label{NS}
\E{ \norm{G}_{C^j(\D)}} \leq c_1(\D) \sigma(\D^{(1)};j+1)  ,
\end{eqnarray}
where $c_1(\D)$ is a constant depending on $\D$.
\end{Prop} 

The following inequality is the so-called {\bf Borell-TIS inequality} applied to the Gaussian field $\partial_{\alpha}G$, see for instance \cite[Theorem 2.1.1]{AT:09}.
\begin{Prop}\label{PropBTIS}
For every $\alpha \in \N^d$ and $u>0$, we have 
\begin{eqnarray}\label{BTIS}
\Prob\{ \sup_{x\in D} \partial_{\alpha} G(x)
> \E{\sup_{x\in \D} \partial_{\alpha} G(x)}
+ u\}
\leq e^{-\frac{u^2}{2 \sigma^2(\D;|\alpha|)}}.
\end{eqnarray}
\end{Prop}

Combining the contents of Propositions \ref{PropNS} and \ref{PropBTIS}, we deduce that for every $\alpha\in \N^d$ with $|\alpha|\leq j$ and $u>0$
\begin{eqnarray*}
\Prob\{\sup_{x\in \D} \partial_{\alpha} G(x) > c_1(\D) \tilde{\sigma}(j+1)+u \}
\leq e^{-\frac{u^2}{2\tilde{\sigma}^2(j+1)}}
\end{eqnarray*}  
which implies (by symmetry)
\begin{eqnarray*}
\Prob\{\sup_{x\in \D} |\partial_{\alpha} G(x)| > c_1(\D) \tilde{\sigma}(j+1)+u \}
\leq 2e^{-\frac{u^2}{2\tilde{\sigma}^2(j+1)}}.
\end{eqnarray*} 
Therefore summing over all possible $\alpha$ with $|\alpha|\leq j$,
\begin{eqnarray}\label{EstA}
\Prob\{\norm{G}_{C^j(\D)} > c_1(\D) \tilde{\sigma}(j+1)+u \}
\leq k(j,d) e^{-\frac{u^2}{2\sigma^2(j+1)}}
\end{eqnarray} 
where $k(j,d):=2\mathrm{card}\{ \alpha\in \N^d: |\alpha|=j\}$.

\medskip
 We can now prove Proposition \ref{Mom}.
\begin{proof}[Proof of Proposition \ref{Mom}]
By stationarity of $G$ it follows that, if $\D'$ is a translation of  $\D$, then necessarily $c_1(\D)=c_1(\D')$, where $c_1(\D)$ is the constant appearing in \eqref{NS}.  In particular, 
applying \eqref{NS} in the case where $\D$ is  a ball $\mathbb{B}$ with unit radius   and exploiting the  moment assumption \eqref{Ass} on $G$, we deduce that 
\begin{eqnarray*}
\E{ \norm{G}_{C^j(\mathbb{B})}} \leq c_1 \tilde{\sigma}(j+1),
\end{eqnarray*}
where $c_1$ is a universal constant. Therefore, applying  \eqref{EstA}  with $\D=\mathbb{B}$ yields
\begin{eqnarray*}
\Prob\{\norm{G}_{C^j(\mathbb{B})} > c_1  \tilde{\sigma}(j+1)+u \}
\leq k(j,d) e^{-\frac{u^2}{2\tilde{\sigma}^2(j+1)}}, \quad u>0.
\end{eqnarray*}
Now, using the above inequality with $u=t-c_1  \tilde{\sigma}(j+1)$, we can write for every $b>0$ (setting $k:=k(j,d), \tilde{\sigma}:=\tilde{\sigma}(j+1)$),
\begin{eqnarray}\label{expballs}
\E{e^{b \norm{G}_{C^j(\mathbb{B})}}}   
&=& 1+ b\int_{0}^{\infty} e^{tb}\Prob\{\norm{G}_{C^j(\mathbb{B})} > c_1  \tilde{\sigma} +(t-c_1  \tilde{\sigma}) \} dt  \notag \\ 
&=& e^{b c_1 \tilde{\sigma}}
+ b\int_{c_1\tilde{\sigma}}^{\infty}
e^{tb}\Prob\{\norm{G}_{C^j(\mathbb{B})} > c_1  \tilde{\sigma} +(t-c_1  \tilde{\sigma}) \} dt  \notag \\ 
&\leq& e^{b c_1 \tilde{\sigma}}
+ bk\int_{c_1\tilde{\sigma}}^{\infty}
e^{tb} e^{-\frac{(t-c_1\tilde{\sigma})^2}{2\tilde{\sigma}^2}} dt 
\leq e^{b c_1 \tilde{\sigma}}
+ bk\int_{\R}
e^{tb} e^{-\frac{(t-c_1\tilde{\sigma})^2}{2\tilde{\sigma}^2}} dt  \notag \\ 
&=&e^{b c_1 \tilde{\sigma}}
+ bk \sqrt{2\pi} \tilde{\sigma}\E{e^{bZ}} , \quad Z \sim \mathcal{N}(c_1\tilde{\sigma},\tilde{\sigma}^2)  \notag \\ 
&=& e^{b c_1 \tilde{\sigma}}+ bk \sqrt{2\pi}  \tilde{\sigma}\left(e^{bc_1\tilde{\sigma}+b^2\tilde{\sigma}^2/2} \right)
= e^{bc_1 \tilde{\sigma}} (1+bk\sqrt{2\pi}  \tilde{\sigma}e^{b^2\tilde{\sigma}^2/2})  \notag \\ 
&\leq& e^{bc_1 \tilde{\sigma} + b^2 \tilde{\sigma}^2/2}(1+bk\sqrt{2\pi}  \tilde{\sigma}) 
\leq e^{bc_1 \tilde{\sigma} + b^2 \tilde{\sigma}^2/2+bk\sqrt{2\pi}  \tilde{\sigma}}\notag\\
&=& e^{b\tilde{\sigma}(c_1  +k\sqrt{2\pi}  ) + b^2 \tilde{\sigma}^2/2},
\end{eqnarray}
where we used that $1+x\leq e^x$.
Now for $\D\subset \R^d$ we denote by $N_{\D}$ the minimal number  of unit balls needed to cover $\D$ and by $\mathcal{B}_{\D}:=\{\mathbb{B}_1,\ldots, \mathbb{B}_{N_{\D}}\}$ the collection of all unit balls covering $\D$ in such a way that $\mathrm{card}(\mathcal{B}_{\D})=N_{\D}$. Then, we have that, for every $b>0$
\begin{eqnarray*}
\E{ \norm{G}_{C^j(\D)}} &=&
 \E{\log \exp(b^{-1}b\norm{G}_{C^j(\D)} )} 
= b^{-1} \E{ \log e^{b \norm{G}_{C^j(\D)} }} \\
&\leq& b^{-1} \log\E{e^{b\norm{G}_{C^j(\D)}}} 
\leq b^{-1} \log \sum_{l=1}^{N_D} \E{e^{b\norm{G}_{C^j(\mathbb{B}_l)}}} \\
&\leq& b^{-1} \log\( N_{\D} 
\E{e^{b\norm{G}_{C^j(\mathbb{B}_1)}}}\) \\
&\leq&  b^{-1} \log \( N_{\D} e^{b\tilde{\sigma}(c_1  +k\sqrt{2\pi}  ) + b^2 \tilde{\sigma}^2/2} \) \qquad \mathrm{using  \ \eqref{expballs}}\\
&=& b^{-1}\log(N_{\D})+ \tilde{\sigma}(c_1  +k\sqrt{2\pi} )+ b\frac{\tilde{\sigma}^2}{2} =: h(b). 
\end{eqnarray*}
Differentiating $h$ with respect to $b$, we find that $h(b)\leq h(b_0)$ for $b_0=\sqrt{2}\sqrt{\log( N_{\D}})/\tilde{\sigma}$ and thus 
\begin{eqnarray}\label{p1}
\E{ \norm{G}_{C^j(\D)}} \leq h(b_0) =
\sqrt{2}\tilde{\sigma}\sqrt{\log (N_{\D})}+ \tilde{\sigma}(c_1  +k\sqrt{2\pi} ) =: \mu .
\end{eqnarray}
Now let $p\geq 1$. Then, using  the  inequality 
\begin{eqnarray*}
\Prob\{\norm{G}_{C^j(\D)} > \mu+ u \} \leq 
ke^{-\frac{u^2}{2\tilde{\sigma}^2}}, \quad u>0
\end{eqnarray*}
together with \eqref{p1}, yields
\begin{eqnarray*}
\E{ \norm{G}_{C^j(\D)}^p }
&=& p\int_{0}^{\infty} t^{p-1} 
\Prob\{\norm{G}_{C^j(\D)} > \mu+ (t-\mu) \}dt \\
&\leq & \mu^p + pk\int_{\mu}^{\infty}
t^{p-1}e^{-\frac{(t-\mu)^2}{2\tilde{\sigma}^2}} dt 
\leq \mu^p + pk\int_{\R}
|t|^{p-1}e^{-\frac{(t-\mu)^2}{2\tilde{\sigma}^2}} dt\\
&=&\mu^p + pk \sqrt{2\pi}\tilde{\sigma}\E{|Z|^{p-1}}, \qquad Z\sim\mathcal{N}(\mu,\tilde{\sigma}^2).\end{eqnarray*}
Now for $Z\sim\mathcal{N}(\mu,\tilde{\sigma}^2)$ and $Z':= (Z-\mu)/\tilde{\sigma} \sim \mathcal{N}(0,1)$,
\begin{gather*}
\E{|Z|^{p-1}} = \tilde{\sigma}^{p-1} \E{|Z'+\mu/\tilde{\sigma}|^{p-1}} \leq 2^{p-2}\tilde{\sigma}^{p-1}\( \E{|Z'|^{p-1}}+ (\mu/\tilde{\sigma})^{p-1}\)
=: C_p  (\tilde{\sigma}^{p-1}+\mu^{p-1}),
\end{gather*}
where $C_p:= 2^{p-2} \E{|Z'|^{p-1}}$ depends only on $p$, so that 
\begin{eqnarray*}
\E{ \norm{G}_{C^j(\D)}^p } 
\leq \mu^p + pk \sqrt{2\pi}\tilde{\sigma}C_p  (\tilde{\sigma}^{p-1}+\mu^{p-1}).
\end{eqnarray*}
The conclusion  follows from the definition of $\mu$ in  \eqref{p1} and the fact that there are constants $C_1,C_2>0$ such that $C_1 \mathrm{vol}(\D)\leq N_{\D} \leq C_2\mathrm{vol}(\D)$.
\end{proof}

\end{appendix}

\section*{Acknowledgements}

We thank Maurizia Rossi for several fruitful discussions. GP is partially supported by the FNR grant HDSA (O21/16236290/HDSA) at Luxembourg University. AV is supported by the co-financing of the European Union - FSE-REACT-EU, PON Research and Innovation 2014-2020, DM 1062/2021.

\bibliographystyle{alpha}
\bibliography{Bibliography}

\end{document}